\documentclass{article}
\usepackage[british]{babel}
\usepackage{amsmath,amssymb,xcolor,latexsym,theorem,bbm}
\usepackage{xcolor}
\usepackage{epigraph}
\definecolor{blue}{rgb}{0,0,1}
\definecolor{red}{rgb}{1,0,0}
\def\red{\begin{color}{red}}
\def\ered{\end{color}}
\usepackage{times}
\usepackage[textsize=footnotesize,color=blue!40, bordercolor=white]{todonotes}

\usepackage{lineno}
\usepackage[pagebackref,bookmarks]{hyperref} 
\hypersetup{pagebackref=true,bookmarks=true,
       hyperfigures=true,
	unicode=true,
	colorlinks=true,
	citecolor=blue,
	linkcolor=blue,
	anchorcolor=red
}

\newcommand{\textdots}{...}
\newcommand{\tmcolor}[2]{{\color{#1}{#2}}}
\newcommand{\tmem}[1]{{\em #1\/}}

\newcommand{\tmop}[1]{\ensuremath{\operatorname{#1}}}
\newcommand{\tmrsub}[1]{\ensuremath{_{\textrm{#1}}}}
\newcommand{\tmrsup}[1]{\textsuperscript{#1}}

\newcommand{\tmtextbf}[1]{\text{{\bfseries{#1}}}}
\newcommand{\tmtextit}[1]{\text{{\itshape{#1}}}}
\providecommand{\qed}{{QED}}

\newtheorem{lemma}{\sc{Lemma}}
\numberwithin{lemma}{section}
\newtheorem{corollary}[lemma]{\sc{Corollary}}
\newtheorem{definition}[lemma]{\sc Definition}
{\theorembodyfont{\rmfamily}}
{\theorembodyfont{\rmfamily\small}}
\numberwithin{exercise}{section}

{\theorembodyfont{\rmfamily}}
{\theorembodyfont{\rmfamily}}
\newtheorem{theorem}[lemma]{\sc Theorem}

\def\ed{\end{document}}

\def\me{{\mathsf{E}}}

\def\mi{{\mathsf{I}}}
\def\mj{{\mathsf{J}}}

\def\mo{{\mathsf{o}}}
\def\ed{\end{document}}
\renewcommand{\qed}{\hfill QED}

\begin{document}


\newcommand{\hi}{version Sep 23 15a)}
\newcommand{\abo}[2]{\ensuremath{#1 <_{\mathcal{O}}   #2}}
\newcommand{\calo}{ {\mathcal{O}}}
\newcommand{\calm}{\ensuremath{\mathcal{M}} }
\newcommand{\cali}{\ensuremath{\mathcal{I}} }
\newcommand{\calj}{{\mathcal{J}}}
\newcommand{\calk}{ {\mathcal{K}}}
\newcommand{\call}{ {\mathcal{L}}}
\newcommand{\calc}{C}
\newcommand{\gott}{\ensuremath{\mathfrak{T}}}
\newcommand{\GRP}{\ensuremath{(\ensuremath{\operatorname{GRP}}_{0} )}}
\newcommand{\LST}{\ensuremath{\call_{\dot{\in}}}}
\newcommand{\ent}{{{\em Entscheidungsproblem\/}}}
\newcommand{\bara}{\bar{{\alpha}}}
\newcommand{\bp}{\bar{{\pi}}}
\newcommand{\ts}{\ensuremath{T^{\ast}}}
\newcommand{\tsptwo}{\left.{\emp}\right) \
\ensuremath{_{\textrm{p\tmrsub{2}}}}}
\newcommand{\tptwo}{\ensuremath{T^{\ast} ( \varnothing )_{p_{2}}}}
\newcommand{\pe}{\ensuremath{P^{\ensuremath{\operatorname{\mathsf{eJ}}}}_{e}}}
\newcommand{\tmt}[1]{\ensuremath{\mathbbm{T}^{#1}}}
\newcommand{\p}[1]{\ensuremath{P^{\ensuremath{\operatorname{\mathsf{eJ}}}}_{#1}}}
\newcommand{\zx}[1]{\ensuremath{{^{#1}\zeta}}}
\newcommand{\Mx}[1]{\ensuremath{{^{#1} \!M}}}
\newcommand{\sx}[1]{\ensuremath{{^{#1} \Sigma}}}
\newcommand{\tx}[1]{\ensuremath{^{#1} T}}
\newcommand{\cst}{\ensuremath{\macro{C^{\ast}}}}
\newcommand{\tmcr}{\ensuremath{\curvearrowright}}
\newcommand{\cl}{\ensuremath{\curvearrowleft}}
\newcommand{\ou}{{\"o}}
\newcommand{\au}{{\"a}}
\newcommand{\uu}{{\"u}}
\newcommand{\go}{G{\"o}del}
\newcommand{\phiav}[1]{\ensuremath{\Phi^{\alpha}_{#1}}}
\newcommand{\phiaf}{\ensuremath{\Phi_{\ensuremath{\operatorname{fp}}}^{\alpha}}}
\newcommand{\phiabp}[2]{\ensuremath{\Phi_{\beta +1}^{(  \alpha )}}}
\newcommand{\phiab}[2]{\ensuremath{\Phi_{#2}^{#1}}}
\newcommand{\phivf}[1]{\ensuremath{\Phi_{\ensuremath{\operatorname{fv}}}^{#1}}}
\newcommand{\rar}{{\rightarrow}}
\newcommand{\midd}{ {\mid} }
\newcommand{\con}{{\hspace{-0.03em}\smallfrown \hspace{-0.03em}}}
\newcommand{\lt}{{\leq}\ensuremath{_{\textrm{T}}}}
\newcommand{\qu}{Q}
\newcommand{\pb}{\ensuremath{\Psi_{\beta}}}
\newcommand{\oy}{<\ensuremath{_{\textrm{y\tmrsub{0}}}}}
\newcommand{\oyr}{\ensuremath{>_{y_{0}}}}
\newcommand{\oye}{{\leq}\ensuremath{_{\textrm{y\tmrsub{0}}}}}
\newcommand{\ptwo}{{\mbox{{\em II\/}}}}
\newcommand{\pone}{{\mbox{{\em I\/}}}}
\newcommand{\uhr}{{\upharpoonright}}
\newcommand{\2}{{{\em II\/}}}
\newcommand{\1}{{{\em I\/}}}
\newcommand{\wff}{\ensuremath{\ensuremath{\operatorname{wff}}(y_{0} )}}
\newcommand{\nod}{{\noindent}}
\newcommand{\diam}{\ensuremath{}{\diamond}}
\newcommand{\lalp}{ {\pa{L\ensuremath{_{\textrm{{\alpha}}}},{\in}}}}
\newcommand{\cit}{{\macro{}{{\cite{\arg{k}}}}}}
\newcommand{\ul}{\ensuremath{\ulcorner}}
\newcommand{\ur}{\ensuremath{\urcorner}}
\newcommand{\cs}{\ensuremath{\ul \sigma \ur}}
\newcommand{\ct}{\ensuremath{\ul \tau \ur}}
\newcommand{\cso}{\ensuremath{\ul \sigma_{0} \ur}}
\newcommand{\co}[1]{\ensuremath{\ulcorner #1 \urcorner}}
\newcommand{\gp}[1]{\ensuremath{\prec #1 \succ}}
\newcommand{\tmfi}{\ensuremath{\ensuremath{\operatorname{Field}}( \prec )}}
\newcommand{\al}{\ensuremath{\aleph}}
\newcommand{\alo}{\ensuremath{\aleph_{\omega}}}
\newcommand{\cub}{c.u.b.}
\newcommand{\dff}{\ensuremath{\ensuremath{\operatorname{Def}}}}
\newcommand{\bu}{\ensuremath{\bullet}}
\newcommand{\nat}{\ensuremath{\mathbbm{N}}}
\newcommand{\re}{\ensuremath{\mathbbm{R}}}
\newcommand{\baire}{\ensuremath{\mathbbm{N}}\textsuperscript{\ensuremath{\mathbbm{N}}}}
\newcommand{\bai}{\textsuperscript{\ensuremath{\omega}}\ensuremath{\omega}}
\newcommand{\linf}{{\leq}\ensuremath{_{\textrm{{\infty}}}}}
\newcommand{\linfs}{<\ensuremath{_{\textrm{\text{{\tmstrong{\ensuremath{\infty}}}}}}}}
\newcommand{\leinf}{{\leq}\ensuremath{_{\textrm{\text{{\tmstrong{\ensuremath{\infty}}}}}}}}
\newcommand{\cant}{2\textsuperscript{\ensuremath{\mathbbm{N}}}}
\newcommand{\can}{\textsuperscript{\ensuremath{\omega}}2}
\newcommand{\rest}{{\upharpoonright}}
\newcommand{\emp}{{\varnothing}}
\newcommand{\pa}[1]{\ensuremath{\langle #1 \rangle}}
\newcommand{\wst}{\ensuremath{S}}
\newcommand{\ws}[2]{\ensuremath{S}\textsuperscript{#1}\ensuremath{_{\textrm{#2}}}}
\newcommand{\wsone}[1]{\ensuremath{S}\textsuperscript{1}\ensuremath{_{\textrm{#1}}}}
\newcommand{\wstwo}[1]{\ensuremath{S}\textsuperscript{2}\ensuremath{_{\textrm{#1}}}}
\newcommand{\wthree}[1]{\ensuremath{S}\textsuperscript{3}\ensuremath{_{\textrm{#1}}}}
\newcommand{\es}[2]{\ensuremath{E}\textsuperscript{#1}\ensuremath{_{\textrm{#2}}}}
\newcommand{\esone}[1]{\ensuremath{E}\textsuperscript{1}\ensuremath{_{\textrm{#1}}}}
\newcommand{\estwo}[1]{\ensuremath{E}\textsuperscript{2}\ensuremath{_{\textrm{#1}}}}
\newcommand{\wet}{\ensuremath{E}}
\newcommand{\game}{\ensuremath{\Game}}
\newcommand{\games}{\ensuremath{\Game \Sigma^{0}_{3}}}
\newcommand{\szero}{\ensuremath{\Sigma^{0}_{3}}}
\newcommand{\hright}{{\hookrightarrow}}

\newcommand{\callt}{L\ensuremath{_{\textrm{T}}}}
\newcommand{\calt}{T}
\newcommand{\cald}{D}
\newcommand{\calu}{\ensuremath{\mathcal{U}}}
\newcommand{\lst}{\ensuremath{\mathcal{L}_{\dot{\in}}}}
\newcommand{\calw}{W}
\newcommand{\calv}{V}
\newcommand{\vp}{\ensuremath{\varphi}}
\newcommand{\om}[1]{\ensuremath{\omega_{1} ( #1 ) }}
\newcommand{\da}{{\downarrow}}
\newcommand{\ua}{{\uparrow}}
\newcommand{\la}{{\langle}}
\newcommand{\ra}{{\rangle}}
\newcommand{\ran}{\text{ran}}
\newcommand{\dom}{\text{dom}}
\newcommand{\back}{{\backslash}}
\newcommand{\power}{{\mathcal{P}}}
\newcommand{\mar}{{\marginpar{\ensuremath{\rightarrow}}}}
\newcommand{\bm}{{\boldmath}}
\newcommand{\opistol}{\text{\ensuremath{O^{\pistol}}}}
\newcommand{\ie}{{\itshape{i.e.}}{\hspace{0.25em}}}
\newcommand{\eg}{{\itshape{e.g.}}{\hspace{0.25em}}}
\newcommand{\etc}{{\itshape{etc.}}{\hspace{0.25em}}}
\newcommand{\via}{{\itshape{via }}{\hspace{0.25em}}}
\newcommand{\cf}{{\itshape{cf. }}{\hspace{0.1em}}}
\newcommand{\mm}{{\itshape{m.m.}}{\hspace{0.25em}}}
\newcommand{\an}{{\wedge}}
\newcommand{\bigconj}{\ensuremath{\bigwedge \hspace{-1.0em} \bigwedge}}
\newcommand{\tmk}{\text{``}}
\newcommand{\proves}{{\vdash}}
\newcommand{\all}{{\forall}}
\newcommand{\Equi}{{\,\Longleftrightarrow\,}}
\newcommand{\equi}{{\,\longleftrightarrow\,}}
\newcommand{\ex}{{\exists}}
\newcommand{\Imp}{{\,\Rightarrow\,}}
\newcommand{\sset}{ { \,\subseteq\, } }
\newcommand{\lr}{{\leftrightarrow}}
\newcommand{\imp}{{\,\longrightarrow\,}}
\newcommand{\ri}{{\,\rightarrow\,}}
\newcommand{\wt}{\tilde{\}}}
\newcommand{\finold}{\ensuremath{[ \alpha ]^{< \omega}}}
\newcommand{\rem}{{\noindent}{\bfseries{Remark: }}}
\newcommand{\fin}[1]{\ensuremath{[  ]^{< \omega}}}
\newcommand{\dlimtwo}{\ensuremath{\raisebox{-1.25\ensuremath{\operatorname{ex}}}{\overset{Lim}{\longrightarrow}
\hspace{0.25em}}}}
\newcommand{\dlim}{\text{{\raisebox{-1.25ex}{\ensuremath{\overset{Lim}{\longrightarrow}
\hspace{0.25em}}}}}}
\newcommand{\qedtwo}[1]{\hfill Q.E.D.(#1)}
\newcommand{\pf}{{\noindent}{\textbf{Proof: }}}
\newcommand{\dfs}{\ensuremath{=_{\ensuremath{\operatorname{df}}}}}
\newcommand{\edfs}{\ensuremath{  \Equi_{\ensuremath{\operatorname{df}}}  }}
\newcommand{\pr}{\ensuremath{\prec}}
\newcommand{\pre}{\ensuremath{\preceq}}
\newcommand{\js}{\ensuremath{J_{s}}}
\newcommand{\jsb}{\ensuremath{J_{\bar{s}}}}
\newcommand{\jst}[3]{ \ensuremath{\langle J^{#1}_{#2} , #3 \rangle}}
\newcommand{\jnb}{J\ensuremath{_{\textrm{\bar{{\nu}}}}}}
\newcommand{\jn}{J\ensuremath{_{\textrm{{\nu}}}}}
\newcommand{\lam}[2]{\ensuremath{\Lambda ( #1 , #2 )}}
\newcommand{\nb}{\ensuremath{\bar{\nu}}}
\newcommand{\nbn}{\ensuremath{\bar{\nu} \prec \nu}}
\newcommand{\ns}{{\nu}\ensuremath{_{\textrm{s}}}}
\newcommand{\tmsb}{\ensuremath{\bar{s}}}
\newcommand{\nsb}{\ensuremath{\nu_{\bar{s}}}}
\newcommand{\maps}{\ensuremath{f: \bar{s} \Longrightarrow s}}
\newcommand{\mapnr}{\ensuremath{f \Longrightarrow \nu}}
\newcommand{\mapn}{\ensuremath{f: \bar{\nu} \Longrightarrow \nu}}
\newcommand{\mapg}[3]{\ensuremath{#1 : \overline{#2} \Longrightarrow #3}}
\newcommand{\onektuple}[1]{\ensuremath{#1_{1} , \ldots ,  #1_{k}}}
\newcommand{\zeroktuple}[1]{\ensuremath{#1_{0} , \ldots  , #1_{k}}}
\newcommand{\onentuple}[1]{\ensuremath{#1_{1} , \ldots  , #1_{n}}}
\newcommand{\zerontuple}[1]{\ensuremath{#1_{0} , \ldots  , #1_{n}}}
\newcommand{\ntuple}[2]{\ensuremath{#1_{1} , \ldots  , #1_{#2}}}
\newcommand{\otuple}[2]{\ensuremath{#1_{0} , \ldots  , #1_{#2}}}
\newcommand{\signk}{\ensuremath{\Sigma_{k}^{(n)}}}
\newcommand{\signo}{\ensuremath{\Sigma_{1}^{(n)}}}
\newcommand{\bsigno}{\ensuremath{\ensuremath{\boldsymbol{\Sigma}}_{1}^{(n)}}}
\newcommand{\bsigmo}{\ensuremath{\ensuremath{\boldsymbol{\Sigma}}_{1}^{(m)}}}
\newcommand{\sigmk}[1]{\ensuremath{\Sigma_{#1}^{(n)}}}
\newcommand{\sigmnk}[2]{\ensuremath{\Sigma_{#2}^{( #1 )}}}
\newcommand{\sigms}{ \ensuremath{\Sigma^{\ast}}}
\newcommand{\bsigms}{{\textbf{ \ensuremath{\Sigma^{\ast}}}}}
\newcommand{\lnn}{ \ensuremath{\Lambda (q, \nu )}}
\newcommand{\lns}{ \ensuremath{\Lambda (q,s)}}
\newcommand{\lnsp}{ \ensuremath{\Lambda^{+} (q,s)}}
\newcommand{\lnsl}{ \ensuremath{\Lambda (q,s| \lambda )}}
\newcommand{\fb}[1]{\ensuremath{\lambda (f_{( #1 ,q, \nu )} )}}
\newcommand{\fbb}[2]{\ensuremath{\lambda (f_{( #1 ,q,  #2 )} )}}
\newcommand{\fbs}[1]{\ensuremath{\lambda (f_{( #1 ,q,s)} )}}
\newcommand{\lis}{\ensuremath{l^{i}_{s}}}
\newcommand{\lhs}{\ensuremath{l^{i}_{\eta s}}}
\newcommand{\lls}{\ensuremath{l^{i}_{\lambda s}}}
\newcommand{\llf}{\ensuremath{\lambda = \lambda (f)}}
\newcommand{\csp}{C\ensuremath{_{\textrm{s}}}\textsuperscript{+}}
\newcommand{\csl}{C\ensuremath{_{\textrm{s|{\lambda}}}}}
\newcommand{\rn}{ \ensuremath{\rho^{n+1}_{M} > \kappa}}
\newcommand{\rhn}[1]{{\rho}\textsuperscript{#1}\ensuremath{_{\textrm{M}}}}
\newcommand{\rhnm}[2]{{\rho}\textsuperscript{#1}\ensuremath{_{\textrm{#2}}}}
\newcommand{\eqn}{{\omega}{\rho}\textsuperscript{n+1}\ensuremath{_{\textrm{M}}}{\leq}{\kappa}<{\omega}{\rho}\textsuperscript{n}\ensuremath{_{\textrm{M}}}}
\newcommand{\hnm}{H\textsuperscript{n}\ensuremath{_{\textrm{M}}}}
\newcommand{\oddpagetext}[1]{\newcommand{\pageoddheader}{{\small }}}
\newcommand{\evenpagetext}[1]{\newcommand{\pageevenheader}{{\small }}}



\title{\large  A Turing machine model for Kleene Type 2 recursion}
\author{\normalsize P.~D. Welch\\ \small{School of Mathematics,}\
\small{University of Bristol}}
\date{\small\today}
\maketitle
\abstract{We give an account of Kleene's Type 2 recursion theory modelled on Turing machines.  We apply this account to observe that the feedback computation of \cite{AFL2020} is an example of Kleene Recursion in $^2\mathsf{E}$. An application of Gandy Selection in the feedback setting solves questions there raised on uniformly finding indices for recursive unions {\em etc.}  of feedback semi-computable sets; further it allows for effective choice and other principles. \footnote{
{\em MSC2020 Mathematics Subject Classification:} 03D65, 03D10, 03D75, 68Q04, 68Q10.
}

\epigraph{Dedicated to the memory of \linebreak Robin Oliver Gandy (1919-1995)}

\section{\normalsize Introduction}

\subsection{\normalsize Type 2 Kleene Recursion}

The purpose of this paper is to give a brief overview of this kind of recursion in a {\tmem{type 2
functional}}, that is some {\em total } $\mathsf{I}:\,^{\omega} \omega \imp \omega$.  This will be done in a Turing machine theoretic setting, as opposed to the equational calculus setting that was favoured in the 1960s and 1970s.
Kleene introduced recursion for all type-n functionals in a series of papers, but here we shall stick to type 2.  

We then relate the notion of {\em feedback computation} of \cite{AFL2020}  to Kleene's type 2 recursion in $\me = \mathsf{^2 \mathsf{E}}$, (see Def. \ref{DefE} below) showing that the feedback computation is a special case of Kleene recursion in $\me$. But also, that Kleene recursion in $\me$ can, in return, be modelled as feedback computations. It should be remarked that Kleene Recursion was of course much wider, being intended to be applicable for any functional $\mathsf{I} :^{\omega}\!\! \omega \imp \omega$, rather than just $\me$.

The {\em loci classici} for Kleene's  higher type recursion were his papers  {\cite{Kl59}} and {\cite{Kl63}},  which chronologically bookended two further papers {\cite{Kl62b}}, {\cite{Kl62a}} . The papers {\cite{Kl59}} and
{\cite{Kl63}} exposited an equational based theory of generalised recursion in
all finite types, and it was these and their definitions that were mostly
taken up by researchers in the following years, notably Gandy, Grilliot, Harrington, Normann, Moschovakis,
{\tmem{et al.}} In the intermediate papers \ {\cite{Kl62b}}
, {\cite{Kl62a}} Kleene gave an alternative Turing machine based picture of
these finite type recursions and showed that the classes of
functions computed by the equational and machine conceptions to be the same. Kleene states that he was trying to show  the naturalness of his definitions by finding
equivalent models leading to the same computation classes, \via Turing machines
and \via equational calculi, just as Turing had done for his machines. However Kleene's
Turing machine picture did not gain much prominence.

Here in this note everything is restricted to Type 2 and their functionals
such as $\mathsf{I}$ above. A very readable exposition of the equational
calculus version can be found in Hinman, Part VI of {\cite{Hi78}}. Kleene
showed the many basic properties that Type-2 recursion in a simple functional
$\mathsf{E}$ enjoyed: the sets Kleene recursive in $\mathsf{E}$ coincided with
the hyperarithmetic sets, the Kleene semi-recursive in $\mathsf{E}$ sets were
precisely the $\Pi^1_1$ sets, and a complete Kleene semi-recursive set of integers
was therefore a complete $\Pi^1_1$ set. The picture was presumably most
satisfying. These developments, as well as later work, can all be found in
Hinman.

The equational approach can seem forbidding in its austerity. Gandy acknowledges this in his \cite{Ga67}. He suggests that much of the development can be regarded as plainly part of ordinary recursion theory, and indeed suggests a more perspicuous approach for the reader would be to use the Register Machines of Shepherdson and Sturgess   
\cite{SheStu63}. He proceeds there to outline how this would look - again performing this for all finite types.

Here we sketch a Turing machine equivalent for type 2 alone. This should really be read as a version of Kleene's model from \cite{Kl62b}, shorn of all the typing variables and of any generality.  The basic picture here is  that
of a Turing machine which can consult an oracle in the form of a type 2
functional such as $\mathsf{I}$ above. The obvious question is ``how?'' given
that by its nature a Turing machine is a finitary object and only deals in a
finite amount of information. (The same question arises in a finitary function
equational calculus. Kleene, as we do here, posited that the objects for the computation all ``lay outside" the machine: there are no registers containing infinitary objects.  These objects are consulted as oracles.) The main idea is that the $e$'th procedure, or
computation recursive in a type 2 functional $\mathsf{I}$, $P_e^{\mathsf{I}} ( \vec{m}, \vec X)$ (with inputs $\vec{m} \in \, ^{< \omega} \omega$ and $\vec X
\in \, ^{< \omega} (^{\omega} \omega)$) is  just as in ordinary
Turing recursion with real oracle(s) $\vec X$. (We are going to simply
ignore all niceties about the differences between elements of $^{\omega}
\omega$ and $^{\omega} 2$: we leave it to the reader to choose their favourite
way of computably coding integers, and sequences thereof, into elements of
\tmrsup{$\omega$}2 and shall just speak of machines acting on integers rather
than just elements of $2.$)  However the essential feature is that a {\em query} or {\em  subcomputation call}, of the
form $? Q ( e', \vec{m}', \vec X ) ?$ can be made at any stage of the computation (with
integers $e', \vec m'$ being passed to the subcomputation, the oracle $\vec X$ being
fixed throughout). This will have the effect of asking for an
$\omega$-sequence of computations $P_{e'}^{\mathsf{I}}$($\vec m',
\vec X$, $k$) $\downarrow y (k)$ to be made.

If all of these computations converge for $k = 0, 1, 2, \ldots$ we have a real
$y = (y (0), y (1) \ldots)$ which can be presented to the oracle $\mathsf{I}$.
It is the integer value $\mathsf{I} (y)$ that is finally returned to the
master computation which can then proceed with it. We regard this pause for
the Query $Q$ and to receive its response, as just one step, or stage in the calling
master computation. It is convenient to arrange a plan of the whole
computation as a downwards growing tree $\frak{T} =\mathfrak{T}^{\mathsf{I}} (e,
\vec m, \vec X)$ with the master computation at the topmost node,
which infinitely splits downwards into the $\omega$ nodes calculating
$P_{e'}^{\mathsf{I}}$($\vec m', \vec X$, $k$) below it, each tree $\mathfrak{T}^{\mathsf{I}} (e',
(\vec m',k), \vec X)$ being naturally a subtree of $\frak{T}$.

The master computation $P_e^{\mathsf{I}} (\vec m, \vec X)$ can fail
to converge with an output for several reasons: a subcomputation call may
simply not return a value $\mathsf{I} (y)$. Either because some
$P_{e'}^{\mathsf{I}}$($\vec m', \vec X$, $k$) has {\tmem{diverged}},
$P_{e'}^{\mathsf{I}}$($\vec m', \vec X$, $k$){$\uparrow$}, because the
Turing machine running at this node does not halt but runs through infinitely
many stages. Then there is no value $y (k)$ and hence no $y$ that can be
submitted to the oracle $\mathsf{I}$. Or, it could be that because the tree $\frak{T}^\mi(e,\vec m,\vec X)$ is
illfounded: a subcomputation call may have an infinite descending chain of
subcomputation calls, or queries, below it on the tree, each waiting for some information to be passed
up from a lower subcomputation so that it might continue. Finally it could be because the master computation itself may have its
node at the top of a wellfounded tree, but may simply run through $\omega$
stages without halting. Kleene showed that the convergent computations
$P_e^{\mathsf{E}} (\vec m, \vec X)$ all had their trees
$\mathfrak{T}^{\mathsf{E}} (e, \vec m, \vec X)$ as wellfounded
elements of $L_{\omega_1^{\tmop{ck}, \vec X}} [\vec X]$, that is
$\Delta^1_1 (\vec X)$, and thus of tree rank less than
$\omega_1^{\tmop{ck}, \vec X}$. 

Here $\mathsf{E}$ is the type 2 existential number quantification operator:
\begin{definition}
  \label{DefE} (The functional $ \mathsf{^2 \mathsf{E}}$) \
  (i)  We define $\mathsf{E} = \mathsf{^2 E}:  \bai \imp 2$
  by:
  \[ {\mathsf{E}} (X) = \left\{ \begin{array}{ll}
       0 & \text{if } X \in \: \bai \wedge \ex n X (n) = 0\\
       1 & \text{} \text{if } X \in \: \bai \wedge \all n X (n) \neq 0.
     \end{array} \right. \]
\end{definition}

Kleene called operators $\mj$ in which $\me$ was Kleene recursive (``$\me \leq \mj$") {\em normal}, as this ensured good behaviour of computations. We shall consider only normal operators here.

We can thus give a ranking of successful
computations by the ordinal rank of these trees. We prefer to give an equivalent
computation rank through the definition of a monotone operator that builds up
the class of successful computations in a functional $\mathsf{I}$. This will
be done uniformly in $\mathsf{I}$ in Section \ref{GSFC}.

The advantages of the Turing machine model are apparent when it comes to proving some of the basic theorems. Our proof of the Stage Comparison Theorem below is relatively brief compared to that dealing with the ``16 cases plus two `otherwise' cases" of \cite{Hi78} p.286. Further, substitution theorems come to be  straightforward with a Turing machine style argument. We have in any case avoided excessive formalism when laying down the basic definitions, and only given relatively free descriptions of some arguments.

\subsection{\normalsize Feedback Computation }

In {\cite{AFL2020}} (see also {\cite{AFL2019}} and {\cite{AFL2015}}) the
authors introduce the concept of {\tmem{Feedback (Turing) Computation}}. The
key point is that during a feedback Turing computation a special oracle can be
consulted which answers queries of the form: $? Q${\tmem{ Does the Turing
computation $\{ e \}^X (m) $ converge or diverge?}} and receive a $0 / 1$
answer in the form of $\uparrow$ or ${\downarrow}$. Here $X \in \:^{\omega} 2$
is the usual oracle for computability relative to the set/oracle $X$.
Consequently such a computation can formulate queries about computations it
wishes to simulate and know convergence/divergence facts about them. But the
real point is that the queries are really to sub-computational processes with the same
architecture: $? Q${\tmem{ Does the {\tmem{feedback}} Turing computation $\la
e \ra^X (m) $ converge or diverge?}} (They use the angle bracket notation
$\pa{e}^X$ for feedback computations in $X$.) Again, as for Kleene recursion, the machine itself contains no infinitary objects: but here, and this is an important point, the query can be thought of as asking whether a certain course-of-computation is finite or not, and receiving an answer.

We may envisage the query as calling the subcomputation $\la e \ra^X (m)$ and
running the subcomputation, which in turn may call further subcomputations of
itself{\textdots}  Consequently such a feedback computation results (when
successful) in a wellfounded tree of subcomputations. The authors show that
there is the very close connection with hyperarithmetic and $\Pi^1_1$ sets.
The former end up being the feedback computable and the latter the feedback
semi-computable sets (\cite{AFL2020}, Prop. 3.5). The wellfounded trees arising for successful
computations of this form, are all hyperarithmetic and so have some (ordinary
Turing) computable ordinal rank less than $\omega^{\tmop{ck}}_{1 }$ ({\em op.cit.}, Prop. 2.6).\quad A
computation $\la e \ra^X (m)$ resulting in an illfounded tree is naturally one
in which calls to increasing depth of subcomputations are made along the
rightmost path of that tree. This consists of subcomputations waiting for
$\downarrow / \uparrow$ information to be passed up, and are essentially
{\em frozen} as they term it. (Freezing, or being undefined, is denoted $\la e \ra^X (m) \Uparrow$). If the tree is
wellfounded ($\la e \ra^X (m) \Downarrow$) the computation may nevertheless
still diverge, as the top level master computation may never halt.\\

It is natural to ask about the relationships of this form of computation with
the classical Kleene recursion in $\mathsf{E}$ given the close agreement
between outputs: Kleene showed that the Kleene semi-recursive in the
functional $\mathsf{E}$ sets were precisely the $\Pi^1_1$ sets, with the
Kleene recursive sets being the hyperarithmetic sets, with the same results
for the trees of computations. In fact the agreement is tight. Feedback
recursion can be seen as an example of Kleene type 2 recursion in $\mathsf{E}
:$

\begin{theorem}
  \label{ftok}There is a primitive recursive function $k_1$ so that for any
  $e, m$ and $X$:
  
  {\hspace{10em}}$\pa{e}^X (m) \simeq \{ k_1 (e) \}^{\mathsf{E}} (m, X) .$
\end{theorem}

In fact the converse also holds: Kleene recursions in $\mathsf{E}$ can be
simulated in a uniform way by feedback machines,

\begin{theorem}
  \label{ktof}There is a primitive recursive function $k_0$ so that for any
  $e, m, X$
  
  {\hspace{10em}}$\{ e \}^{\mathsf{E}} (m, X) \simeq \pa{k_0 (e)}^X (m)$.
\end{theorem}

Consequently results about Kleene recursion in $\mathsf{E}$ can be translated
across to feedback recursion. In particular the feedback semi-computable sets
of integers, will be precisely the Kleene semi-recursive in $\mathsf{E}$ sets,
in short they are the $\Pi^1_1$ sets, as was already shown in {\cite{AFL2020}}.

\

In Section \ref{GSFC} we address one such question that arose in {\cite{AFL2020}} at their Cor. 3.6.

\begin{definition}
  {\em{\cite{AFL2020}}} (Def 2.8) Let $X, Y \in \, 2^{\omega}$. Then $X$ is
  {\tmem{feedback reducible to $Y$}}, or \tmem{feedback} \tmem{computable}
  from $Y, X \leq_F Y$, if there is $e \in \omega$ such that:
  
  (i) $\all n \in \omega${\pa{e}}$^Y (n) \Downarrow$
  
  (ii) $\all n \in \omega${\pa{e}}$^Y (n) \downarrow$
  
  (iii) $\all n \in \omega${\pa{e}}$^Y (n) = X (n)$.
\end{definition}

In particular then, a set $B \sset \omega$ is then {\tmem{feedback
computable}} if the characteristic function $K_B \in 2^{\omega}$ is feedback
computable with an empty oracle $Y$: \ $\ex e \in \omega ( \all n
K_B (n) = \la e \ra (n)$). We call such an $e$ an {\tmem{index}} of such a
feedback computable set.

\begin{definition}
  {\em\cite{AFL2020}} (Def 3.5).\quad A set $B \sset \omega$ is {\tmem{feedback
  semi-computable (in $X$)}} when there is an $e \in \omega$ such that $n \in
  B \Equi \pa{e}^X (n) \Downarrow$.
\end{definition}

Similarly we call $e$ a (semi-computable) index for $B$. Note that the
requirement is simply that the computation $\pa{e}^X (n)$ is defined, that is
has a wellfounded tree of subcomputations. We can thus think of the
semi-computable in $X$ sets as the domain of the partial feedback computable
function $\pa{e}^X (n)$.
They then show that a $B \sset \omega$ is feedback semi-computable in $X$ if
and only if $B$ is $\Sigma_1 (L_{\omega^X_{1 \tmop{ck}}} [X])$ where as usual
$L_{\omega^X_{1 \tmop{ck}}} [X]$ is the least transitive admissible set
containing $X$, where in turn $\omega^X_{1 \tmop{ck}}$ is the least ordinal
not (Turing) recursive in $X$.\quad Hence, as is well known, this means $B$ is
feedback semi-computable in $X$ iff it is $\Pi^1_1 (X)$.

From that characterisation much can follow.

\begin{corollary}
  {\em\cite{AFL2020}}(Cor 3.6) A set $B \sset \omega$ is feedback computable if
  and only if both it and its complement $\neg B$ are feedback
  semi-computablee.
\end{corollary}

Further:
\begin{corollary}
  The feedback semi-computable sets are closed under finite unions and
  $\ex^{\leq \omega}$ existential quantification.
\end{corollary}

They then ask: \\

\nod
{\em Question: Given indices $e_0, e_1$ for sets $B,
\neg B$ which are both assumed semi-computable (that is $B = \left\{ n \mid
\pa{e_0} (n) \Downarrow \right\}$ and $\neg B = \left\{ n \mid \pa{e_1} (n)
\Downarrow \right\}$, is there a uniform method for finding an index for $e =
e (e_0, e_1)$ witnessing that $B$ is feedback computable?}\\

Note that the naive approach does not work for this question: one could devise
a putative procedure $\pa{e}$ that simulates simultaneously \ $\pa{e_0} (n)$
and $\pa{e_1} (n)$ until one or other converges. But it may be that $\pa{e_0}
(n) \downarrow$ whilst $\pa{e_1} (n) \Uparrow$ and the latter may occur before
the former convergence. Then the whole procedure $\pa{e}$ becomes undefined.

\

However, now with Theorem \ref{ftok} we can drag the question over to one about $k_0
(e_0)$ and $k_0 (e_1)$ and Kleene recursive sets. In this arena we do have
that there is a p.r. function $g$ so that for any pair of indices $(e_0',
e'_1)$ for complementing feedback semi-computable sets $B', \neg B',$ $g
(e'_0, e'_1)$ returns an index for what is now the Kleene recursive set $B'$.
By Theorem \ref{ktof} $k_1^{- 1} (g (k_0 (e_0), k_0 (e_1)))$ is then an index for $B$
as a feedback computable set.

\

A similar question (and answer) would then be, given indices $e_0, e_1$ as
above for semi-computable sets $B, C$, is there a uniform method for finding
an index $e$ from $e_0$ and $e_1$ for $B \cup C$?

\

This is perhaps a little unsatisfying, and in Section \ref{GSFC} we prove
 there are uniform methods for directly answering these questions, by proving a
version of the Gandy Selection Theorem {\cite{Ga67a}} tailored for feedback
computation. To do this we prove a Stage Comparison Theorem for feedback
computation (Section \ref{GSCT} Theorem \ref{OrdComp}) which is of independent (and fruitful) interest in its own right.
As the authors of {\cite{MY1996}} write, this approach of Gandy to first
prove a Stage Comparison Lemma, and from that a Selection Theorem and then
other consequent results, became the standard methodology in higher type and
generalised recursion theory, ``and of its offshoots in descriptive set theory". This approach first appeared (without proofs) in Gandy's \cite{Ga67a} paper.

\

Lastly we observe that the definition of feedback semi-computable sets, as
domains of partial feedback computable sets which have wellfounded computation
trees, {\ie} using the relation $\Downarrow$, coincides with the definition
made using ${\downarrow}$.

\begin{lemma}
  There is a primitive recursive function $g$ so that $$\all e, m \all X \pa{g (e)}^X (m)
  \downarrow \Equi \pa{e}^X (m) \Downarrow.$$
\end{lemma}

{\pf}Given $e$ the feedback computation $g (e)$ does the following: on input
$m$ it asks for the subcomputation that answers the query ``?{\tmem{Does the
computation}} $\pa{e}^X (m)$ $\downarrow$ or $\uparrow ?$''. Then $\la g (e)
\ra^X (m)$ converges with an answer, whatever it is, precisely when $\pa{e}^X
(m) \Downarrow$. {\qed}

\

\

Consequently the feedback semi-computable (in $X$) sets (using $\Downarrow$)
can be defined as the class of sets which are domains of partial
semi-computable (in $X$) functions with convergence $\downarrow$. And these
are precisely the Kleene semi-recursive (in X) and $\mathsf{E}$ sets.

\section{\normalsize Definitions and Stage Comparison for Kleene Recursion}

We give here the definition of Kleene recursion in an arbitrary type-2 functional $\mi$, and prove Gandy's Stage Comparison Theorem for it, in Section 2.1. In 2.2 we prove Gandy's Selection Theorem.

\subsection{\normalsize Kleene recursive functions}
\begin{definition}
  \label{3.8} For $\mi$ any type-2 functional, we set $\Gamma^{\mathsf{I}} (C) =$:
  \[ \begin{array}{l}
     \hspace{-1em}  \{ \langle \pa{e, \vec m, \vec X}, n \rangle |
       P_e^{\mathsf{I}} (\vec m, \vec X) \text{ is a Kleene
       computation in $\mathsf{I}$ making only oracle calls }\\
       \quad \quad \quad\hspace{5.17em} Q^{\mathsf{I}} (e', \vec m', \vec X)
       \hspace{0.17em} \hspace{0.17em} \text{and receiving back }  n'
       = \mathsf{I} (y)   \text{ where }  \\  \hspace{9em}  y = \lambda k.y(k) \mbox{ and } 
       \all k \in \omega   \la \pa{e',
       (\vec m', k), \vec X}, y (k) \ra \in C \left. \, \right\}.
     \end{array} \]
  {\nod} As $\Gamma^{\mathsf{I}}$ is monotone, we may let
  
  $$\Gamma^{\mathsf{I}}_0 = \emp; \Gamma_{< \alpha}^{\mathsf{I}} =
  \bigcup_{\beta < \alpha} \Gamma_{\beta}^{\mathsf{I}} \,\wedge \,
  \Gamma_{\alpha}^{\mathsf{I}} = \Gamma (\Gamma^{\mathsf{I}}_{< \alpha})$$ in
  the usual way, and reach a least fixed point $\Gamma^{\mathsf{I}}_{\infty}.$
\end{definition}

\begin{definition}
  The {\tmem{rank}} of a defined computation, $\rho^{\mathsf{I}} (
  \langle \pa{e, \vec m, \vec X}, n \rangle )$ is the least
  $\beta$, if it exists, such that $\langle \pa{e, \vec m, \vec X},
  n \rangle \in \Gamma_{\beta}^{\mathsf{I}}$ for some (unique) $n$. We often
  abbreviate this as $\rho^{\mathsf{I}} (e, \vec m, \vec X)$ with
  $n$ understood but unspecified.
\end{definition}

Thus we rank convergent computations by that $\alpha$ where they appear in the
inductive definition of convergent computations.
If $\rho^{\mi}(e, \vec m, \vec X)$ is undefined (because $(e, \vec m, \vec X)\notin dom(\Gamma^\mi_\infty)$) then we instead set it to be $\omega_1$. 
Then:

\begin{definition}[\bf The \boldmath{$\{ e \}^{\mathsf{I}}$}'th function partial
Kleene recursive in {\boldmath{$\mathsf{I}$}}]\mbox{}\\
  Using $\Gamma_{\infty}^{\mathsf{I}}$:
  
   $\{ e \}^{\mathsf{I}} (\vec m, \vec X)$\\
  (i) is {\tmem{defined}},
  or {\tmem{convergent, with output}} $n$ iff $\Gamma_{\infty}^{\mathsf{I}}
  (\pa{e, \vec m, \vec X}) = n$ with $n \in \omega$.
  In which case we set $ {\{ e \}^{\mathsf{I}}}  (\vec m, \vec X) =
  n$ or write $\{ e \}^{\mathsf{I}} (\vec m, \vec X) \downarrow n$, or simply 
  $\{ e \}^{\mathsf{I}} (\vec m, \vec X)\da$;\\
  \nod (ii) is
  {\tmem{undefined}}, or {\tmem{divergent}} written $\{ e \}^{\mathsf{I}}
  (\vec m, \vec X) \uparrow$, when $\pa{e, \vec m, \vec X}
  \notin \dom (\Gamma^{\mathsf{I}}_{\infty})$.\\
  (iii) $\{ e \}^{\mathsf{I}} $ is {\tmem{Kleene recursive in $\mathsf{I}$}} \
  if it is partial Kleene recursive in $\mathsf{I}$ and total.
\end{definition}

Having given the formal definitions, we are going to relax our notation, and
in the sequel blithely write ``$m, X$'' for ``$\vec m, \vec X$''
pretending we have single integer and real inputs, trusting that the reader
can make the appropriate amendments if they feel it needed, as in the next definition.

\begin{definition}
 (i) a) The function $f \sset \omega \times \omega$,  or the functional $F:  (\omega \times \bai) \imp \omega$, is {\tmem{ Kleene partial
  recursive in}} $\mathsf{I}$  if it is $\{ e \}^{\mathsf{I}}$ for some
  $e \in \omega$.\\
  b) A function $f \in\,^{\omega}\! \omega$ or functional $F:  (\omega \times \bai) \imp \omega$ is {\tmem{Kleene
  recursive in $\mathsf{I}$}} \ if it is Kleene partial recursive in
  $\mathsf{I}$ and total, and so with domain $\omega$ ($\omega \times \bai$ resp.) \\ c) A relation $R$ is
  {\tmem{Kleene recursive in}} $\mathsf{I}$ if its characteristic function
  $K_R$ is Kleene recursive in $\mathsf{I}$.
  
  (ii) a)  A relation $R$ is {\tmem{Kleene semi-recursive in}} $\mathsf{I}$ if it
  is the domain of a function Kleene partial recursive in $\mathsf{I}$. \\ b) It is
  {\tmem{Kleene co-semi-recursive in $\mathsf{I}$}}, if its complement is
  {\tmem{{\tmem{Kleene semi-recursive in $\mathsf{I}$}}.}}
\end{definition}

There are many basic lemmata to be proven about this definition. We list a few.

\begin{lemma} For any $\mj:$ the class of relations semi-recursive in $\mj$ is closed under universal number quantification, $\all^0$.
\end{lemma}
\begin{lemma}\label{Karithmetical}
 Any arithmetical relation $R$ 
 is recursive in $\me$.
 \end{lemma}

\begin{lemma} 
For any $\mj$, the class of functionals partial recursive in $\mj$ is closed under functional substitution, \ie, if $G$ and $H$ are partial recursive in $\mj$, and
$$ F(\vec m,\vec X)\simeq G(\vec m,\vec X, \lambda p.H(p,\vec m,\vec X)),$$
then $F$ is p.r. in $\mj$.
\end{lemma}
\pf For $Y$ a real variable, suppose and $H(p,\vec m,\vec X))= \{h\}^\me(p,\vec m,\vec X)$ and $G(\vec m,\vec X,y) = \{g\}^\me(\vec m,\vec X,Y) =P_g^\me(\vec m,\vec X,Y)$. We modify $P_g$ to obtain a suitable $P_f$. Whereas $P_g$ has inputs $\vec m$ on tape and use of oracle calls to $\vec X$ and $Y$, instead of an oracle call $?Y(p)=?$ we have in the modified program $P_f$ a subroutine that instead computes the value of $\{h\}^\me(p,\vec m,\vec X)$ and places that at the relevant point of the tape. (Note that if $\{h\}^\me(p,\vec m,\vec X)$ is not total in $p$ then $F(\vec m,\vec X)\simeq P_f(\vec m,\vec X)$ is undefined.)\qed\\

One would then proceed in a familiar manner and prove an $S^n_m$-Theorem, and thence:

\begin{theorem}[Recursion Theorem for Kleene Recursion]\label{KRT}

For any functional $F(e,\vec m, \vec X)$ Kleene partial  recursive in $\mi$, there exists $\bar{e}\in \omega$ such that $\all \vec m, \vec X\, F(\bar e,\vec m, \vec X)=\{\bar e\}^\mi(\vec m, \vec X)$. \end{theorem}

\begin{lemma}\label{Kconvergence}
There is an index $g$ for a partial Kleene recursive  function $\{ g \}=\{ g \}^{\me} $ that
effects the following:\\ (1)
 $\{ g \}  ( (\bar{e}, \vec{m}), \vec X)\ua$
 iff $\frak{T}^\me(\bar e,\vec m,\vec X)$ is illfounded. Otherwise:\\
(2) $\{ g \}  ( (\bar{e}, \vec{m}), \vec X)\da
0/1$ if and only if $\{\bar e\}^\me(\vec m, \vec X)\da /\ua$.

\end{lemma}
\pf
$\{ g_0 \}^\me  ( ((\bar{e}, \vec{m}),k), \vec X)$ simulates a run of computation of $\{
\bar e \} ^\me ( \vec{m}, \vec X) $ for $k$ steps.
If after the $k$ steps are completed this simulation of $\{
\bar e \}  ( \vec{m}, \vec X) $ has halted, then we specify that $\{ g_0 \}^\me  (
((\bar{e}, \vec{m}),k), \vec X)\da 0$. Otherwise we specify that \linebreak  $\{ g_0 \}  ^\me( ((\bar{e},
\vec{m}),k), \vec X)\da 1$. Now let $\{g\}^\me$ be the procedure that launches the single query $?Q(g_0, (\bar e ,   \vec{m}), \vec X)$ and halts with output  $\me(\lambda k.\{ g_0 \} ^\me (((\bar{e}, \vec{m}),k), \vec X))$. Then $\{g\}^\me$ satisfies the requirements of the lemma.

Note firstly that the definition of $\{g\}^\me$ is independent of, or is uniform in, $\bar e$, and secondly that $\{g\}$ will be divergent with illfounded computation tree, exactly when  $\{\bar e\}^\me(\vec m, \vec X)$ has such.
\qed\\

In the sequel we shall let greek letters such as $\gamma,\delta$ stand in for 
computations and simply write $\gamma$ for $\{e\}^{\mj}(\vec m, \vec X)$
We let $\gamma_p  (p < k (\gamma))$ (where $k (\gamma) \leq \omega$)
enumerate the query calls $? Q^{\mj} (e_p, \vec m_p, \vec X) ?$ occurring in turn in the run of
a computation $\gamma$ as $\gamma_p = (e_p, \vec m_p, \vec X)$. Of course if $k (\gamma) =
\omega$ then $\gamma\uparrow\!.$

\begin{lemma}
  \label{subcalls} (i) There is an index function $f_0$ so that $$\{f_0 \} ^{\me}(k,e,\vec m,\vec X)\downarrow 
    \pa{e_k, \vec m_k, \vec X} \mbox{ for } k < k(e,\vec m,\vec X)$$ where the $k$'th query call
 is  $? Q (e_k,\vec m_k, \vec X) ?$
   during a  run (possibly divergent) of $\{e\}^\me(\vec m,\vec X)$. In the notation above if $\gamma = (e, \vec m,\vec X)$ then  for $k < k(\gamma)\leq\omega$, 
   $$\{f_0\}^\me(k,\gamma) =  
    \gamma_k .$$
  (ii) There is a variant index $f_1$ so that  $f_1(k,\gamma) = \{f_0\}^\me(k,\gamma)= \gamma_k$ for $k<k(\gamma)$ whilst $f_1(k,\gamma)= \gamma_{k(\gamma)-1}$ if $\omega> k\geq k(\gamma)$.
\end{lemma}

Given a subcomputation call $\gamma_p=   \pa{e_p, \vec m_p, \vec X} \mbox{ for } p < \omega$ of some $\gamma$ we additionally write $\gamma_{p,i} $ for the $i$'th calculation
$\{e_p\}^\me((\vec m_p,i),\vec X))$. Then note that $\rho(\gamma)=\sup^+\{\rho(\gamma_{p,i})\mid i,p<\omega\}$. We shall use this  in the next theorem.

\begin{theorem}[(Gandy) Stage Comparison for Kleene Recursion]  \label{OrdComp}\mbox{} \\ Suppose $\me$ is Kleene recursive in $\mj$. There is a  Kleene recursive function $G=G^\mj$, such
  that for any  $\gamma = \pa{e^0, m^0, X^0}$
  and $\delta = \pa{e^1, m^1, X^1} $ with subcomputation calls $\pa{\gamma_p\mid p<k(\gamma)}$ and $\pa{\delta_p\mid p<k(\delta)}$:\\
  
  (i) $\rho (\gamma) < \omega_1 \wedge \rho (\gamma) \leq \rho  (\delta) \imp
  G (\gamma, \delta) \simeq 0 ;$
  
  (ii) $\rho (\delta) < \omega_1 \wedge \rho (\delta) < \rho (\gamma) \imp G
  (\gamma, \delta) \simeq 1.$
\end{theorem}

{\pf} 
Define a partial function $F$ as follows:\\

\nod$
\begin{array}{rcl}

 F(e,\gamma, \delta) = &0&\mbox{ if }\rho ( \gamma ) =0\\

&1&\mbox{ if }\rho ( \gamma ) \neq 0
\wedge   \rho ( \delta ) =0 \\

&1&\mbox{ if } \ex \gamma_{k,i} \all \delta_{l,j}
\{e\}^\mj(\gamma_{k,i}, \delta_{l,j}) \downarrow 1;\, (A)\\

 & 0 &\mbox{ if } 
 { \all \gamma_{k,i} \ex  \delta_{l,j}  
  \{e\}^\me (\gamma_{k,i}, \delta_{l,j}) \downarrow 
 0}. \,(B)\\

\end{array}
$\\

\nod{\tmem{Claim} $F$ is partial recursive in $\mj  $}.

{\pf}of {\em Claim}. 
The first two lines in the definition of $F$ by cases, are predicates
arithmetic in $X$ and so by Lemma \ref{Karithmetical}, they are Kleene recursive in $\mj$. Suppose  then both $\rho (
\gamma ) , \rho ( \delta ) >0$. 
We set $\me^\mo(X) =1\dot - \me(\lambda p.[1\dot - X(p)])$ to be the dual of $\me$. Then we have:\\

$
\begin{array}{rcl}
\me^\mo(X) & = &0, \mbox{ if } \all j X(j)=0;\\
&=&1, \mbox{ if } \ex j X(j)\neq 0.
\end{array}
$\\

\nod Let $\pi((u)_0,(u)_1)=u$ be some ordinary recursive bijective pairing function $\pi:\omega\times \omega \equi \omega$ with unpairing functions $(u)_0,(u)_1$.
Lines (A) and $(B)$ above then can be expressed together as 
$$
F(e,\gamma,\delta)  = \me^\mo(\lambda p. \me(\lambda q. \{e\}^\mj(\gamma_{(p)_0,(p)_1}, \delta_{(q)_0,(q)_1})).
$$
As $\me$ is recursive in $\mj$, so is $\me^\mo$, and hence so is $F$.
{\qed} Claim\\

By the Recursion Theorem \ref{KRT} there is $\bar{e}$ with $\{\bar{e}\}^{\mj}
( u,v ) =F ( \bar{e} ,u,v )$, and we set
$$G ( u,v ) = \{\bar e\}^\mj ( u,v ).$$

\nod We have that:\\

\nod{\tmem{Claim 1}}

$
\begin{array}{lcl}

G ( \gamma , \delta ) = &0&\mbox{ if }\rho ( \gamma ) =0\\

&1&\mbox{ if }\rho ( \gamma ) \neq 0
\wedge   \rho ( \delta ) =0; \mbox{ otherwise: }\\

&1&\mbox{ if } \ex \gamma_{k,i} \all \delta_{l,j} G (
\gamma_{k,i} , \delta_{l,j} ) \simeq 1; (A)\\

 & 0 &\mbox{ if } 
{ \all  
\gamma_{k,i}  \ex \delta_{l,j} G ( \gamma_{k,i} , \delta_{l,j} ) \simeq 0}; (B)\\

\end{array}
$\\

We now prove that $G$ satisfies the statements  {\tmem{(i)}} and {\tmem{(ii)}} of the Theorem, by induction on
$\sigma \dfs \min \{ \rho ( \gamma ) , \rho ( \delta ) \}$ $ < \omega_{1}$. (If
both $\rho ( \gamma ) = \rho ( \delta ) = \omega_{1}$ there is nothing to do.)
Thus at least one of these ranks is countable.

$\sigma =0:$ then either $\rho ( \gamma ) =0$ and $G ( \gamma , \delta ) =0$
or $0= \rho ( \delta ) < \rho ( \gamma )$ and so $G  ( \gamma , \delta ) =1$.

$0< \sigma$: as inductive hypothesis we assume{\tmem{ (i)}} and {\tmem{(ii)}}
hold for all computation pairs $\bar{\gamma} , \bar{\delta}$ with $\min \{ \bar{\gamma} ,
\bar{\delta} \} < \sigma$.

Suppose first that $\rho ( \gamma ) < \omega_{1} \wedge \rho ( \gamma ) \leq
\rho ( \delta )$. Then $$\all i,j,k,l \: \min \{ \rho^{} (
\gamma_{k,i} ) , \rho ( \delta_{l,j} ) \} < \sigma,$$ and the inductive hypothesis for {\tmem{(i)}}
and {\tmem{(ii)}} applies and we therefore have:\\

(1) $\rho ( \gamma_{k,i} ) \leq \rho ( \delta_{l,j} ) \imp G ( \gamma_{k,i} ,
\delta_{l,j} ) \simeq 0;$

(2) $\rho^{} ( \delta_{l,j} ) < \rho ( \gamma_{k,i} ) \imp G ( \gamma_{k,i} ,
\delta_{l,j} ) \simeq 1.$\\

As these antecedents are mutually exclusive, we have that $\all i,j,k,l \:
 $ $G ( \gamma_{k,i} , \delta_{l,j} ) \da$. As $\rho^{} ( \gamma )
\leq \rho ( \delta )$ by assumption, we have $\all k\all i \ex l \ex jG (
\gamma_{k,i} , \delta_{l,j} ) \simeq 0$. 
Then by (1) and clause $(B)$,  we also have $G ( \gamma , \delta ) =0.$

Suppose secondly, that $\rho ( \delta ) < \rho ( \gamma ) \leq \omega_{1}$.
Then $\all l,j \rho ( \delta_{l,j} ) < \rho ( \delta ) = \sigma$, and by the
inductive hypothesis again on {\tmem{(i)}} and {\tmem{(ii)}} we have that for any $i,j,k,l$, (1) and (2) again hold.

As $\rho ( \delta ) < \rho ( \gamma )$ there are $k_{0}, i_{0} $ least with $\rho (
\delta ) \leq \rho^{} ( \gamma_{k_{0},{i_0}} )$. This implies $\all l,j \;\rho (
\delta_{l,j} ) < \rho ( \gamma_{k_{0},i_0} ) .$ By using (2) again, and the defining
properties of $F$ and $G$ at $(A)$, we have $G ( \gamma , \delta ) =1.$
{\qed}

\subsection{\normalsize Gandy Selection}

This argument is derived from the Kleeneian equational calculus version proof in
\cite{Hi78}
 Part VI.
\begin{theorem}[Gandy Selection for Kleene Recursion]\label{KSelection} For any type 2 functional $\mj$, there exists$\textsf{}$ a
  function $\tmop{Sel}^\mj$ ${\in}^{< \omega} 2$, partial
  recursive  in $\mj$, such that for all $e, \vec m, \vec X$ the following are
  equivalent:
  
  (i) $\ex p \in \omega . \{e\}^\mj ( p, \vec{m}, \vec X ) \downarrow\,$ ;
  
  (ii) $ \{e\}^\mj ( \tmop{Sel}^\mj ( e, \vec{m} ,\vec X), \vec{m}, \vec X
  ) \downarrow\,$.
\end{theorem}

{\pf} Without loss of generality, we take single variables $m,X$. There is a p.r. function $k$ so that for any index $e$, $k(e)  = e^+$ is
an index so that
$$\{e^+ \}^\mj  (p, m, X) \simeq \{e\} ^\mj (p + 1, m,X)\mbox{ for }p \in \omega.$$ 
We
take $G(u, v)$ from the Stage Comparison Theorem \ref{OrdComp}, and define
$F (f, p, m, X)$ to be the functional which accords with the following
procedure.

Step 1 It computes $G( \la e, (0, m), X \ra, \pa{f, (e^+, m), X})$.

Step 2 If this value is $0$, then that is the final value for $F (f, p, m,
X)$. If this value is $1$, it then computes $G( \la e,  ( 0, m), X \ra, \la e,
\{ f \}^\mj((e^+, m),X ) + 1, m ,X\ra) $.

Step 3 If this latter value is 0, then that is the final value. If this value
is $1$, it then computes $\{ f \}^\mj ((e^+, m), X) + 1$ as its final value.\\

\noindent By the Kleene Recursion Theorem there is $\bar{f}$ so that $\{ \overline{f} \} ^{\mj} (p,
m,X) \simeq F (\bar{f}, p, m,X)$. We take $\tmop{Sel}^{\mj}  = \{
\overline{f} \}^{\mj} $. We now show that this works.

Let $p (e, m,X)$ be the least $p$ such that $\{e \}^{\mj}(p, m,X)\da$ if there is
such a $p$.  We show the equivalence of {\em(i)} with {\em{(ii)}} by induction on
$  p (e  , m, X)$. {\em(ii)} $ \Imp$ {\em(i)} is immediate.

\nod {\em Case 1} \ $p (e, m, X) = 0$.

Then $\{e\}^\mj  (0, m,X)\da$.
Hence by properties of $G$, for some $i <
2$
$$G( \la e, (0, m), X \ra, \pa{\bar{f}, (e^+, m), X} ) =
i.$$

\nod If $i = 0$: then $\tmop{Sel}^{\mj}  (e, m,X) = 0$ as required.

\nod If $i = 1$: this implies $\rho (  \bar{f}, (e^+, m), X)   < \rho
(e, ( 0, m), X)$ which in turn means that there is a $q$ with $\{f\}^{\mj}
(e^+, m,X) = q$. Again as $\{e\}^\mj  (0, m,X)\da$
we have that for some $j < 2$,
$$G\left( \la e, ( 0, m), X \ra, \pa{e ,( q + 1, m), X} \right) = j.$$

\nod Again if $j = 0$ then $\tmop{Sel}^{\mj}(e, m,X) = 0$ as required.

If $j = 1$ then we have $\rho( e , (q + 1, m), X ) < \rho( e,( 0, m), X )   $ and
thence $\{ e\} ^{\mj} (q + 1, m,X)$. However then $\tmop{Sel}^{\mj}  (e, m,X) \simeq \{
\bar{f}\}^{\mj}  (e^+, m, X) + 1 = q + 1$, as we wanted.

\

\nod {\em Case 2} Now suppose that $p (e, m, X) > 0$. By our definitions $p (e, m, X) = p (e^+,
m, X) + 1$. So we can apply the induction hypothesis to $p (e^+, m, X)$ and
then there will be $q$ such that
$$\{ \bar{f} \} ^\mj (e^+, m,X) \simeq {  \tmop{Sel}^\mj}  (e^+, m,X) = q 
\mbox{ \& }\{ e \}^\mj  (q + 1, m,X) \simeq \{ e^+ \}^\mj  (q, m,X)$$
\noindent and with both the latter defined. In particular $\{ \bar{f} \} ^\mj (e^+, m,X)\da \,
$. By our assumption on $p$ we have 
$\{e \}^{\mj}(0, m,X)\!\not\!\da$,
and thus
 $$\rho (\bar{f}, (e^+, m), X) < \rho (e, (0, m), X) =\omega_1  $$ and thus
$$G\left( \la e, (0, m), X \ra, \pa{e , (q + 1, m), X} \right) = 1.$$
Hence ${ \tmop{Sel} ^\mj}  (e, m,X) = \{ \bar{f} \}^\mj (e^+, m,X) + 1 = q + 1$
again as required for ${ \tmop{Sel} }^\mj $. \\ \mbox{ }{\qed}\\

\subsection{\normalsize Definitions and Stage Comparison for Feedback Computation \label{GSCT}}
We define a monotone inductive procedure for building up the class of
feedback computations computable in some $X : \omega \imp 2$. This will be
done uniformly in $X$, and in fact for all  $X$ at once. The operator $I$ has domain contained
in $\left( \omega^2 \times^{\omega}\!\! \omega \, \right) \times (\{ \downarrow,
\uparrow \} \times \omega + 1)$.

\begin{definition}
  \label{3.8} We set $I (C) =$:
  \[ \begin{array}{l}
    \hspace{-0.7em}   \left\{ \langle \pa{e, m, X}, \pa{j, k} \rangle \right| P_e^X (m)
       \text{ is a feedback-computation making only halting oracle}\\  \hspace{-0.6em}\text{calls }
        Q^X (e', m') \hspace{0.17em} \hspace{0.17em}
       \text{and receiving back }  j' \in \{ \downarrow, \uparrow \} 
       \text{ where } C (\pa{e', m', X}  ) = \pa{j', k'}
       ;\\
       \quad \quad\mbox{ further if }   j' = \downarrow \mbox{ then } P_{e'}^X
       (m') \downarrow k' \in \omega,  \mbox{ if } j' = \uparrow
       \mbox{then } k' = \omega \left. \, \right\}.
     \end{array} \]
  {\nod}As this is monotone, we may let
  
  $I_0 = \emp$; $I_{< \alpha} = \bigcup_{\beta < \alpha} I_{\beta}$ \&
  $I_{\alpha} = I (I_{< \alpha})$ in the usual way. Fixing $X$ and defining
  $I_{\alpha}^X = I_{\alpha} \rest \omega^2 \times \{ X \}$, $I^X$ reaches a
  least fixed point $I^X_{\infty} = I_{\alpha}^X = I_{\alpha + 1}^X$ for some
  least $\alpha < \omega_1$.
\end{definition}

\begin{definition}
  The {\tmem{ (feedback) rank}} of a defined computation:
  $$\rho \left( \langle \pa{e, m,
  X}, \pa{j, k} \rangle \right),$$ is the least $\beta$, if it exists,  such
  that $\langle \pa{e, m, X}, \pa{j, k} \rangle \in I_{\beta}$ for some
  (unique) $\pa{j, k}$. We often abbreviate this as $\rho (e, m, X)$ with
  $\pa{j, k}$ understood but unspecified.
\end{definition}

Then:

\begin{definition}$\mbox{ }$\linebreak{\sc[The \tmtextbf{$\la e \ra^X$}'th function partial feedback
computable in \tmtextbf{$X$}]}\mbox{ }

 \nod Using $I_{\infty}^X$:
  
  (i) $\la e \ra^X (m)$ is {\tmem{defined}}, or {\tmem{convergent}}, with
  output $k$ iff $I_{\infty}^X (\pa{e, m}) = \pa{\downarrow, k}$ with $k \in
  \omega$.
  
  In which case we set $ \la e \ra^X (m) = k$ or write $\la e \ra^X (m)
  \downarrow k$. If $I_{\infty}^X (\pa{e, m}) = \pa{\uparrow, \omega}$ it is
  {\tmem{divergent}}, written $\la e \ra^X (m) \uparrow$.
  
  (ii) If convergent or divergent, we write additionally $\la e \ra^X (m)
  \Downarrow$.
  
  (iii) Otherwise it is {\tmem{undefined}} or {\tmem{freezing}}, if $\pa{e,
  m} \notin \dom (I^X_{\infty})$, written $\la e \ra^X (m) \Uparrow$.
  
  (iv) $\la e \ra^X $is {\tmem{feedback computable in}} $X$ \ if it is
  partial feedback computable in X and total.
\end{definition}

\begin{definition}  
  The function $f \sset \omega \times 2$ is {\tmem{partial feedback computable in}}
  $X$  if it is $\la e \ra^X$ for some $e \in \omega$. $f$ is {\tmem{feedback computable in $X$}}
if it is    $\la e \ra^X$ for some $e \in \omega$ which is feedback computable in $X$.
 A relation $R$ is
  {\tmem{feedback computable in}} $X$ if the characteristic function $K_R$ is
  feedback computable in $X$.
  
  (ii)  A relation $R$ is {\tmem{feedback semi-computable in}} X if it is the
  domain of a function feedback partial computable in $X$.
\end{definition}

\begin{definition}
  \label{fsemic}{\em\cite{AFL2020}} (Def 3.4) A set $B \sset \omega$ is
  {\tmem{feedback semi-computable in $X$ if there is an $b \in \omega$ so that
  $n \in B \equi \pa{e}^X (b) \Downarrow$.}}
\end{definition}

The reader will note that this definition is saying that the f.~semi-c. sets
are those associated with partial f.c. functions and the non-freezing part of
their domains, not the part of the domains where $\pa{e}^X (b) \downarrow$,
that is, simply converge. We address this below

\begin{definition}
  {\em\cite{AFL2020}} (Def 2.8) For two functions $X, Y : \omega{\imp}2$ we say
  $X \leq Y$ (``$X$ is {\tmem{feedback computable from}} ${Y}$'')  if
  there is $e \in \omega$ so that X$= \pa{e}^Y$. To be explicit this
  means:
  
  $\all n \in \omega :$\quad both $\pa{e}^Y (n) \Downarrow$ and $\pa{e}^Y (n)
  \downarrow$; and\quad$X (n) = \pa{e}^Y (n) .$

\end{definition}
The basic notions and definitions above all come, with some minor changes of notation \etc, from \cite{AFL2020}. We now make some comments on some immediate consequences of those definitions, firstly on the notion of feedback semi-computable sets and their domains of definition.
\begin{lemma}
  There are prim. rec. functions $k_0, k_1$ so that:
  
   (i)
$  {\all}e{\all}m{\pa{e}}^X (m) \Downarrow \equi \pa{k_0 (e)}^X (m)
  \downarrow$.
  
  (ii) ${\all}e{\all}m{\pa{k\tmrsub{1}(e)}}^X (m) \Downarrow \equi \pa{e}^X
  (m) \downarrow$.
\end{lemma}

{\pf}For $k_0$: given $e$ let $P_{k_0 (e)} (m)$ be the procedure that starts
by issuing the query ?$Q^X (e, m) ?$ For those $m$ so that there is a \
response to the query $P_{k_0 (e)}$ outputs a 1. For the others, namely for
those $m$ with $\pa{e}^X (m) \Uparrow,$ $P_{k_0 (e)} (m)$ will receive no
response to the query, and so will diverge.

For $k_1$: similarly, given $e,$ let $P_{k_1 (e)} (m)$ be the procedure that
also starts by issuing the query ?$Q^X (e, m) ?$ If the response is
$\downarrow$ then $P_{k_1 (e)} (m) \downarrow 1$; but if the response is
$\uparrow$, $P_{k_1}$ launches a piece of self-referential code which freezes
on all inputs. Thus $P_{k_1 (e)} (m) \Uparrow$.{\qed}\\

This shows that although the authors of  {\cite{AFL2020}} used being the non-freezing part of the domain of a partial function to define semi-computable relations, one could equally well have used the convergent part. The following is then immediate from the lemma.

\begin{corollary}
  A set $B \sset \omega$ is {\tmem{feedback semi-computable in $X$ if and only
  if there is an $e \in \omega$ so that $n \in B \Equi \pa{e}^X (n)
  \downarrow$.}}
\end{corollary}

From now on we shall abbreviate ``feedback computable'' by ``f.c.'' and \
``feedback semi-computable'' by ``f.semi-c.''\\

We shall also remark that the feedback recursive relations easily include the arithmetic sets.
\begin{lemma}\label{arithmetic} Any $A\sset \, ^n\omega$ which is arithmetic (in $X$) is f.~computable (in $X$).
\end{lemma}
\pf Clearly any (Turing) recursive in $X$ set $A$ is f.c in $X$. It thus suffices to show that the f.c in $X$ relations are closed under existential number quantification $\ex^\omega$.
Let $R\sset ^{n+1}\omega$ be f.c. in $X$.  Then its characteristic function $K_R = \pa{e_R}$ for some $e_R\in\omega$. We should like $K_S$, the characteristic function of $S=\{\vec m\mid \ex n R(n,\vec m)\}$, to be f.c. in $X$. 

Let $\pa{e}^X(\vec m)$ be the f.c. in $X$ process that successively computes $K^X_R(0,\vec m)$, $ K^X_R(1,\vec m), \ldots
, K^X_R(k,\vec m), \ldots$ and which halts if it reaches an $n$ so that \linebreak$K^X_R(n,\vec m)\da
0$. Let $\pa{\bar e}^X(\vec m)$ then be the procedure with the query \linebreak $?Q \mbox{ Does }\pa{e}^X(\vec m)\da / \ua ?$ Identifying $\da$ with $0$ and $\ua$ with $1$, we have that $\pa{\bar e}^X(\vec m)$ is such a characteristic function.
\qed\\

{\em What do we do when we simulate?}

The computation $\pa{e_{0}}^{X} ( m_{0} )$ runs as a Turing machine, using
just the pure TM instruction set until there is a halting oracle query $?Q (
e_{1} ,m_{1} ) ?$ : ``Does $\pa{e_{1}}^{X} ( m_{1} )   \da$ or {$\ua$}?''. The
answer appears {\tmem{deo ex machina}} from the oracle. If we {\tmem{simulate
}}$\pa{e_{0}}^{X} ( m_{0} )  $ during some other \ computation $\pa{e'}^{X} (
k' )$ we do something different: as for an (ordinary) universal Turing
machine, we divide up the infinite tape into some infinite slices, and on one
such slice, we ``run'' a copy of $\pa{e_{0}}^{X} ( m_{0} )$ there. When a
query $?Q ( e_{1} ,m_{1} ) ?$ is encountered on this simulated run, this is to
be answered by an oracle - which in a sense is not simulable. However the
query can also be as a query of the running program of $\pa{e'}^{X} ( k' )$
and we regard $\pa{e'}^{X} ( k' )$ as making it, getting a $\da / \ua$
response (in one step) and passing it to the simulated run of \
$\pa{e_{0}}^{X} ( m_{0} )$.\\

{\bu} We note that although the halting oracle is only returning $\uparrow /
\downarrow$ information to any query, easy exercises show that a lot more can
be gleaned from such queries. Suppose, for example $?Q (e_0, m_0) ?$ returns
$\downarrow$. Then there are other indices $e'_0, e''_0$ effectively obtained
from $e_0$, so that \ $\pa{e_0'}^X (m_0, k) \downarrow$ iff $\pa{e_0}^X (m_0)
\downarrow k$, \ and $P^X_{e''_0} (e_0, m_0)$ has a loop with the queries ?Q
Does $\pa{e'_0}^X (m_0, k) \downarrow$? for an incrementing $k = 0, 1, 2,
\ldots$ (with $k$ kept in another register) until the correct output $k_0$ to
$\pa{e_0}^X (m_0)$ is found. Then by assumption $P_{e''_0} (e_o, m_0)$
discovers the value $k_0$. 
Hence a procedure $P_e^X (m)$ can be written, which when
given the right queries, knows the values of $k$ when $\pa{e'}^X (m') \downarrow
k$ for all its query calls.

We shall in the sequel consider such expanded
procedures with more generalised queries which ``return'' values as well as
$\uparrow / \downarrow$ in this way, without much further comment. As two examples, the next two lemmata are exercises, using variations on the ideas from the previous
paragraph. However, first for the record, proven in the usual manner:

\begin{theorem}[Recursion Theorem for Feedback Computation]\label{FCRT}
For\linebreak any functional $F(e,\vec m, \vec X)$ partial  f.~c., there exists $\bar{e}\in \omega$ such that $$\all \vec m, \vec X\, F(\bar e,\vec m, \vec X)=\pa{\bar e}^{\vec X}(\vec m).$$ \end{theorem}

\begin{lemma}
  \label{L5}There is a p.r. function $f$ so that if $\pa{e}^X(m)\Downarrow$ then  $$\all k{\pa{f (e)}}^X (m,
  k) \downarrow \pa{e_k, m_k}$$ where $? Q (e_k, m_k) ?$ is the $k$'th query
  call during a run of $\pa{e}^X (m)$. (If $\pa{e}^X (m)$ makes only finitely
  many query calls we set $\pa{e_k, m_k} = \pa{e_{k_0}, m_{k_0}}$ where \linebreak$? Q
  (e_{k_0}, m_{k_0}) ?$ is the last query, for $k > k_0$.)
\end{lemma}

\begin{lemma}
  \label{L2.12}There is a p.r. function $g \in \,^{\omega} \omega$ such that
  $$\all e, k, n \in \omega{\pa{g (e)}}^X (m, k, n) \downarrow \equi \pa{e}^X
  (k) = n.$$
\end{lemma}

We adopt the following abbreviating notation: we let greek letters such as
$ \gamma, \delta$ in the sequel stand in for f.~computations $\gamma = \pa{e}^X
(m)$ which we shall also abbreviate as simply $\gamma = (e, m, X)$ or just
$(e, m)$ if the oracle $X$ is understood. If $\gamma =_{\tmop{df}} (e, m, X)$
abbreviates the feedback computation $\la e \ra^X (m)$, which makes a query
call $Q^X (e_0, m_0) ?$ then we say that  $\la e_0 \ra^X (m_0) $ is a
{\tmem{subcomputation}} of \ $\la e \ra^X (m)$, and we abbreviate such a
subcomputation as $\gamma_p =_{\tmop{df}} (e_0, m_0, X)$.

\

We let $\gamma_p  (p < k (\gamma))$ (where $k (\gamma) \leq \omega$)
enumerate the query calls \linebreak$? Q^X (e_p, m_p) ?$ occurring in turn in the run of
a computation $\gamma$ as $\gamma_p = (e_p, m_p)$. Of course if $k (\gamma) =
\omega$ then $\gamma\uparrow\!.$ If we are considering computations relative
to different oracles $X $ and $Y$, such as $\pa{e}^X (m)$ and $\pa{f}^Y (n)$
then we'll write these also as $\gamma = (e, m, X)$ and $\delta = (f, n, Y)$.

\begin{definition}
  We set $\rho^X (e, m) =\rho  (e, m, X)= | (e, m, X) |$ to be the least $\alpha$ such that
  $\la${\pa{e,m,X}}, $\pa{j, k} \ra \in I^X_{\alpha}$ for some $\pa{j, k} .$
\end{definition}

Thus we rank convergent computations by that $\alpha$ where they appear in the
inductive definition of convergent computations. If $(e, m, X) \notin
\tmop{dom} (I_{\infty})$ it is convenient to set $\rho  (e, m, X) = \omega_1$.

\

We then have again the key result here:

\begin{theorem}[Stage Comparison]
  \label{OrdComp} There is a \ f.~s.-c. \ function $G$, such
  that for any $X, Y \in\, ^{\omega} 2$, for all $\gamma = \pa{e^0, m^0, X}$
  and $\delta = \pa{e^1, m^1, Y} :$
  
  (i) $\rho (\gamma) < \omega_1 \wedge \rho (\gamma) \leq \rho  (\delta) \imp
  G (\gamma, \delta) \simeq 0 ;$
  
  (ii) $\rho (\delta) < \omega_1 \wedge \rho (\delta) < \rho (\gamma) \imp G
  (\gamma, \delta) \simeq 1.$
\end{theorem}

{\pf} We shall assume for the sake of brevity that $X = Y$ (and as that is all
we shall use later). The proof when $X \neq Y$ only introduces notational
complexity. Since $X$ will be constant throughout, we drop it from some of our
tuples, leaving its presence understood; thus writing for example $\gamma =
\pa{e^0, m^0}$ rather than $\pa{e^0, m^0, X}$ {\etc} Let $F$ be defined in the following manner.

\nod$
\begin{array}{rcl}

 &0&\mbox{ if }\rho ( \gamma ) =0\\

&1&\mbox{ if }\rho ( \gamma ) \neq 0
\wedge   \rho ( \delta ) =0; \mbox{ otherwise: }\\

F(e,\gamma, \delta) = &1&\mbox{ if } \ex \gamma_{k} \all \delta_{l}
\pa{e}^X(\gamma_k, \delta_l) \downarrow 1;\, (C)\\

 & 0 &\mbox{ if } 
 { \ex \delta_{l} \all  
\gamma_{k}   \pa{e}^X (\gamma_k, \delta_l) \downarrow 
 0}; \,(A)\\

 &0& \mbox{ if }\all \gamma_{n} \ex \delta_{l}
\ex   \gamma_{k} (  \pa{e}^X (\gamma_n, \delta_l) \downarrow   0 \wedge \pa{e}^X (\gamma_k, \delta_l) \downarrow   1 ) ;\, (B)\\
& \ua & \mbox{ otherwise. }

\end{array}
$\\

\nod{\tmem{Claim 1} $F$ is f.~semi-c. in }$X$.

{\pf}of {\em Claim 1}. We give a procedure $P$, feedback computable  uniformly in
$X$, for computing $F$.

The first two lines in the definition of $F$ by cases, are predicates
arithmetic in $X$ and so by Lemma \ref{arithmetic}, they are f.c. in $X$. Suppose then both $\rho (
\gamma ) , \rho ( \delta ) >0$.
We outline the procedure for defining $F$ in these latter cases. We shall ask
if the following overall process $P $ converges or diverges. We are to think of
$P = P^X$ as a f.~semi-c. process, giving rise to a f.c. partial function $\la e'
\ra^{X} ( \gamma , \delta ) $. We specify $P$ as a process that itself
makes query calls, in terms of subprocesses researching queries made by $\gamma$
and $\delta$. \

(1) Write $\delta_{0}$ to the OT.
Search incrementally through $\gamma_{0} , \ldots ,
\gamma_{p_{0}}$ for a $p_{0}$ (if it exists) with the property that
$\pa{e}^{X} ( \gamma_{p_{0}} , \delta_{0} ) \da 1$.
For this once we give some more detail: $P$ simulates $\delta$ until it reaches the first query $\delta_0$ which it may then write to the output tape (without evaluating the query yet). It then simulates $\gamma$ until it reaches the first query there, $\gamma_0$, and may, without evaluating $\gamma_0$,  check if $\pa{e}^{X} ( \gamma_{0} , \delta_{0} ) \ua $ or $\da k \notin 2$. If so then $F ( e, \gamma , \delta )$ is undefined. Otherwise (if $0\neq p_0$),
$\pa{e}^{X} ( \gamma_{0} , \delta_{0} ) \da 0$. Returning to the simulation of $\gamma$, $\gamma_0$ is evaluated and the simulation proceeds, pausing when it reaches the next query $\gamma_1$; again before evaluating this, it checks if $\pa{e}^{X} ( \gamma_{1} , \delta_{0} ) \ua $ or $\da k \notin 2$. And so forth, until it reaches a suitable $p_0$ if the latter exists.

 Hence whilst conducting this search, if we have  for some least $p'$, that \linebreak
$\pa{e}^{X} ( \gamma_{p'} , \delta_{0} ) \ua $ or $\da k \notin 2$, then $F ( e, \gamma , \delta )$ is undefined.

(2) \ If this search does not succeed (\ie does not find any such
$\gamma_{p_0}$, then we are at $(A)$, with $\delta_0$ instantiating $\delta_l$) and ensure $F ( e, \gamma , \delta ) \simeq 0$, and $P$
then halts. If this search  succeeds, (we have then 
$\pa{e}^{X} ( \gamma_{p'} , \delta_{0} ) \da 0$ for $p'<p_0$ but 
$\pa{e}^{X} ( \gamma_{p_{0}} , \delta_{0} ) \da 1$)
then we 
write $\gamma_{p_{0}}$ to the OT
(overwriting $\delta_{0}$). Then $\delta_0$ is evaluated and now:

(3) Similarly search $\delta_{1} , \ldots , \delta_{q_{0}}$ for a $\delta_{q_{0}}$ so
that $\pa{e}^{X} ( \gamma_{p_{0}} , \delta_{q_{0}} ) \da 0$ (if it exists). \ If we come across  some least $q'$, with
$\pa{e}^{X} ( \gamma_{p_0} , \delta_{q'} ) \ua $ or $\da k \notin 2$, then $F ( e, \gamma , \delta )$ is again rendered undefined.

If this search does not halt (so does not find such a $\delta_{q_{0}}$ and we are at  $(C)$ with $\gamma_{p_0}$ instantiating $\gamma_k$ there), then
we ensure $F ( e, \gamma , \delta ) \simeq 1$, halting $P$. If this search
halts, (then
we have $\pa{e}^{X} ( \gamma_{p_0} , \delta_{q'} ) \da 1$ for $q'<q_0$ but 
$\pa{e}^{X} ( \gamma_{p_{0}} , \delta_{q_0} ) \da 0$)
 then
 $\delta_{q_{0}}$ replaces $\gamma_{p_{0}}$ on the OT.

(4) Returning to (1), then $\gamma_{p_{0} +1} , \ldots ,$ is searched for a
$\gamma_{p_{1}}$ so that \linebreak $\pa{e}^{X} ( \gamma_{p_{1}} , \delta_{q_{0}} ) \da
1$.\\

And so forth, zig-zagging back and forth through the query calls of both
computations. We then ask $?Q  \tmop{Does}  P^{X} ( \gamma , \delta ) \da /
\ua$?

If $\ua$ then we are at $(B)$ above, and set $F ( e, \gamma , \delta ) \simeq 0$; if $\da$ then $F (
e, \gamma , \delta )$ has a value according to (2) or (3).
{\qed} {\em Claim 1}\\

By the Feedback Recursion Theorem there is $\bar{e}$ with $\pa{\bar{e}}^{X^{}}
( u,v ) =F ( \bar{e} ,u,v )$, and we set
$$G ( u,v ) = \pa{\bar{e}}^{X} ( u,v ).$$

We now claim that $G$ satisfies (i) and (ii) of the Theorem. We have that:\\

\nod{\tmem{Claim 2}}

$
\begin{array}{lcl}

 &0&\mbox{ if }\rho ( \gamma ) =0\\

&1&\mbox{ if }\rho ( \gamma ) \neq 0
\wedge   \rho ( \delta ) =0; \mbox{ otherwise: }\\

G ( \gamma , \delta ) =&1&\mbox{ if } \ex \gamma_{k} \all \delta_{l} G (
\gamma_{k} , \delta_{l} ) \simeq 1; (C)\\

 & 0 &\mbox{ if } 
{ \ex \delta_{l} \all  
\gamma_{k}  G ( \gamma_{k} , \delta_{l} ) \simeq 0}; (A)\\

 &0&{ \mbox{ if }\all \gamma_{n} \ex \delta_{l}
\ex   \gamma_{k} (  G ( \gamma_{n} , \delta_{l} ) \simeq 0 \wedge G (
\gamma_{k} , \delta_{l} ) \simeq 1 ) .} (B)\\

\end{array}
$\\

We now prove that $G$ satisfies {\tmem{(i)}} and {\tmem{(ii)}} by induction on
$\sigma \dfs \min \{ \rho ( \gamma ) , \rho ( \delta ) \} $ $< \omega_{1}$. (If
both $\rho ( \gamma ) = \rho ( \delta ) = \omega_{1}$ there is nothing to do.)
Thus at least one of these ranks is countable.

$\sigma =0:$ then either $\rho ( \gamma ) =0$ and $G ( \gamma , \delta ) =0$
or $0= \rho ( \delta ) < \rho ( \gamma )$ and so $G  ( \gamma , \delta ) =1$.

$0< \sigma$: as inductive hypothesis we assume{\tmem{ (i)}} and {\tmem{(ii)}}
hold for all pairs $\bar{\gamma} , \bar{\delta}$ with $\min \{ \bar{\gamma} ,
\bar{\delta} \} < \sigma$.

Suppose first that $\rho ( \gamma ) < \omega_{1} \wedge \sigma = \rho ( \gamma ) \leq
\rho ( \delta )$. Then $$\all k<k^{0} \: \all l<k^{1}\min \{ \rho^{} (
\gamma_{k} ) , \rho ( \delta_{l} ) \} < \sigma,$$ and the inductive hypothesis for {\tmem{(i)}}
and {\tmem{(ii)}} applies and we therefore have:\\

(5) $\rho ( \gamma_{k} ) \leq \rho ( \delta_{l} ) \imp G ( \gamma_{k} ,
\delta_{l} ) \simeq 0;$

(6) $\rho^{} ( \delta_{l} ) < \rho ( \gamma_{k} ) \imp G ( \gamma_{k} ,
\delta_{l} ) \simeq 1.$\\

As these antecedents are mutually exclusive, we have that $\all k<k^{0} \:
\all l<k^{1}  $ $G ( \gamma_{k} , \delta_{l} ) \da .$ As $\rho^{} ( \gamma )
\leq \rho ( \delta )$ by assumption, we have $\all k<k^{0} \ex l<k^{1} G (
\gamma_{k} , \delta_{l} ) \simeq 0$. Either there is some $\delta_{l}$ with
$\rho ( \delta_{l} )$ bounding all the $\rho ( \gamma_{k} )$ - in which case
by clause (A) above $G ( \gamma , \delta ) =0$ - or else we have $\sup \{ \rho
( \gamma_{k} ) \mid k<k^{0}=\omega \} = \sup \{ \rho ( \delta_{k} ) \mid k<k^{1}=\omega \} .
$Then by clause $(B)$, \ we also have $G ( \gamma , \delta ) =0.$

Suppose secondly, that $\rho ( \delta ) < \rho ( \gamma ) \leq \omega_{1}$.
Then $\all l<k^{1} \rho ( \delta_{l} ) < \rho ( \delta ) = \sigma$, and by the
inductive hypothesis again on {\tmem{(i)}} and {\tmem{(ii)}} we have that for any $k<k^{0}
,l<k^{1}$, (5) and (6) again hold.

As $\rho ( \delta ) < \rho ( \gamma )$ there is  a $k_{0} <k^{0}$ least with $\rho (
\delta ) \leq \rho^{} ( \gamma_{k_{0}} )$. This implies $\all l<k^{1} \rho (
\delta_{l} ) < \rho ( \gamma_{k_{0}} ) .$ By using (6) again, and the defining
properties of $F,G$ at $(C)$, we have $G ( \gamma , \delta ) =1.$

(More explicitly, but perhaps a  trifle unnecessarily, we can see how the mechanics of this search for $k_0$ plays out in the terms of the procedure $P$:

{\tmem{Claim: \ if $\pa{p_{i} \mid i<n}$ are such that $\gamma_{p_{0}} <
\gamma_{p_{1}} < \cdots$ are chosen in the cycle defining $P$ above, then (a)
$n< \omega$ and (b)}} $p_{n-1} =k_{0}$.

{\pf} Note that if some $p_{j} <k_{0}$, whilst at (3) above, then there is
$\delta_{q_{j}}$ such that $G ( \gamma_{p_{j}} , \delta_{q_{j}} )
\da 0$, and we return to (1) to search for a $p_{j+1} >p_{j}$ with $G (
\gamma_{p_{j+1}} , \delta_{q_j} ) \da 1$. If $p_{j+1} \neq k_{0}$, then there is
$\delta_{q_{j+1}}$ with \ $G ( \gamma_{p_{j+1}} , \delta_{q_{j+1}} ) \da 0$.
However the $p_{j_{i}} <k_{0}$ are increasing, and eventually we reach
$p_{n-1} =k_{0}$.) {\qed}
\subsection{\normalsize Gandy Selection for Feedback Computation} \label{GSFC}

\begin{definition}
  We define the {\tmem{universal f.~semi-computable set}} by: (i) $$U (e,
  \vec m, X)  \equi_{\tmop{df}}{\pa{e}}^X (\vec m) \downarrow$$
  
  (ii) $U^X (e, \vec m)  \equi_{\tmop{df}}$ $U (e, \vec m, X)$ is then
  the {\tmem{universal f.~semi-computable in $X$ set.}} 
\end{definition}

Then our function $\rho$ gives a ranking on $U$, as does the projection
$\rho^X$ on the universal semi-computable in $X$ set $U^X$. In standard
terminology $U^X$ is $\omega${\tmem{-parametrized}} {\via}the first coordinate
$e$. Notice that it is equivalent by our conventions to say $\rho(e,\vec m,X)=\rho^X(e,\vec m)<\omega_1$ or $(e,\vec m)\in U^X$, or $(e,\vec m,X)\in U$.

Let $\Gamma^X$ be the pointclass of f.~semi-computable in $X$ relations in
$^{< \omega} \omega$. We have a norm $\rho^X $on the universal f.
semi-computable in X set $U^X$ and hence on any $A \in \Gamma^X$. We see below
that $\rho^X $is a $\Gamma^X${\tmem{-norm}} in the sense of Moschovakis,
{\cite{Mosch2009}} p.153, or, \ in that there is a f.~semi-computable in $X$
relation $\leq^{\rho}_{\Gamma}$ and a f.~co-semi-computable in $X$ relation
$\leq^{\rho}_{\neg \Gamma}$ with the property that:\\

$
\begin{array}{ll}(\ast)\,\quad\rho^X (y) < \omega_1 \imp & \\ 
\hspace{3em}
   \all X \left\{ [\rho^X (X) <
\omega_1 \wedge \rho^X (X) \leq \rho^X (y)] \equi X \leq^{\rho}_{\Gamma} y
\equi X \leq^{\rho}_{\neg \Gamma} y \right\}. &
\end{array}$

\begin{definition}
  A pointclass $\Gamma$ is said to be {\tmem{normed}} or to have the
  {\tmem{Prewellordering Property }}if every pointset $A \in \Gamma$ admits a
  $\Gamma$-norm.
\end{definition}

\begin{lemma}
  For any function $X \in \:^{\omega} 2$, the class $\Gamma^X$ of relations
  f.~semi-computable in $X$ has the prewellordering property.
\end{lemma}

{\pf}It suffices to show this for the universal f.~semi-c. in $X$ set $U^X
.$We verify $(\ast)$ above for $\rho = \rho^X$ and $\Gamma = \Gamma^X$. Using
the Stage Comparison Theorem \ref{OrdComp} we can take

{\hspace{3em}}$\gamma \leq^{\rho}_{\Gamma} \delta \Equi H (\gamma, \delta) =
0$\quad and\qquad$\gamma \leq^{\rho}_{\neg \Gamma} \delta \Equi H (\gamma,
\delta) \neq 1$.

{\nod}as $H$ is partial f.~computable in $X$, we are done. {\qed}

\

We now prove a key result that delivers many basic results for the f.c. and
f.~semi-c. pointclasses. 
\begin{theorem}[Gandy Selection for Feedback Computation]\label{GSTFC} For any $X \in \,^{\omega} 2$, there exists$\textsf{}$ a
  function $\tmop{Sel}$${{^X}}{\in}^{< \omega} 2$, partial
  feedback semicomputable in $X,$ such that for all $e, m$ the following are
  equivalent:
  
  (i) $\ex p \in \omega . \la e\ra^X \left( p, \vec{m} \right) \downarrow\,$ ;
  
  (ii) $ \pa{e}^X \left( \tmop{Sel}^X \left( e, \vec{m} \right), \vec{m}
  \right) \downarrow\,$.
  
  {\nod}The function $\tmop{Sel}^X$ is uniformly definable in $X$.
\end{theorem}

{\pf} The argument is pretty much that of the Selection Theorem for Kleene recursion Theorem \ref{KSelection} above with the obvious notational adjustments. The proof does not depend intrinsically on the difference between the two modes of oracle usage. Again there is a p.r. function $k$ so that for any index $e$, $k(e)  = e^+$ is
an index so that
$$\pa{e^+ }^X  (p, m) \simeq \pa {e} ^X (p + 1, m)\mbox{ for }p \in \omega.$$ 
We now
take $G(u, v)$ from the feedback computational Stage Comparison Theorem \ref{OrdComp}, and define
$F (f, p, m, X)$ to be the functional which accords with the following
procedure.

Step 1 It computes $G( \la e, 0, m, X \ra, \pa{f, e^+, m, X})$.

Step 2 If this value is $0$, then that is the final value for $F (f, p, m,
X)$. If this value is $1$, it then computes $G( \la e, 0, m, X \ra, \la e,
\pa{ f }^X(e^+, m) + 1, m \ra )$.

Step 3 If this latter value is 0, then that is the final value. If this value
is $1$, it then computes $\pa{ f }^X (e^+, m) + 1$ as its final value.\\

\noindent By the Feedback Computability Recursion Theorem there is $\bar{f}$ so that $\pa{ \overline{f} } ^{X} (p,
m) \simeq F (\bar{f}, p, m,X)$. We take ${{ \tmop{Sel} }}^{X}  = \pa{
\overline{f} }^{X} $. The argument  that this works is as for Kleene Recursion.

Referring back to the proof of Theorem \ref{KSelection}, let $p (e, m,X)$ be the least $p$ such that $\pa{e }^{X}(p, m)\da$ if there is
such a $p$. We show the equivalence of {\em(i)} with {\em{(ii)}} by induction on
$  p (e  , m, X)$. {\em(ii)} $ \Imp$ {\em(i)} is immediate.  
Again we split into cases considering if 
({\em Case 1})  $p (e, m, X) = 0$ or is ({\em Case 2}) $ >0$. We let the reader now check through the argument, which really involves substituting $\pa{e}^X(e,m)$ for $\{e\}^\mj(e,m,X)$ wherever the former or its variants occur; and verifying that expressions such as $G( \la e, 0, m, X \ra, \pa{\bar{f}, e^+, m, X} ) =
i$ are defined for the analagous reasons.
\mbox{ }{\qed}\\

With the  Selection Theorem established we now can use it to get a number of
results about the structure of partial computable functions and
semi-computable sets. The first is immediate from this theorem.

\begin{lemma}
  \label{existclosure}For any relation $R \textsf{}$ that is semi-computable
  in $\textsf{}$X, for any $m$ we have:
   $$\ex q R (q, m) \equi R (\tmop{Sel}^X (m), m).$$ 
\end{lemma}

It might be $\textsf{}$tempting to claim that the union of two semi-computable
in $X$ relations, $R$ and ${S}$ say, is semi-computable in $X$, is
established by running a procedure $P$ that simulates both functions
$\pa{e}^X$ and $\pa{f}^X$ simultaneously whose domains are $R$ and $S$
respectively until, if possible, an $m$ falls into one domain or the other (or
neither). But this will not work since we may have $m \in \tmop{dom} \left(
\pa{e}^X \right) \back \tmop{dom} \left( \pa{f}^X \right)$, but the
computation $\pa{f}^X (m)$ being simulated is not just divergent, but has
$\mathfrak{T}^X (f, m)$ illfounded. Moreover this illfoundedness is revealed
by the combined process freezing before $m$'s membership in $\dom \pa{e}^X$ is
revealed. \ This would render $P$ undefined. The use of the Selection Theorem
neatly gets around this.

\begin{lemma}
  \label{closure}
  
  The class of relations f.~semi-computable in $X$ is closed under: (a) finite
  unions; (b) bounded existential, and existential number quantification `$\ex
  n$'; (c) definitions by cases.
\end{lemma}

{\pf}(a) Suppose $R = \tmop{dom} \left( \pa{e}^X \right)$ and $S = \tmop{dom}
\left( \pa{f}^X \right)$ are two semi-computable relations on $\omega$.

Let $F$ be the (ordinary) Turing function defined by $F (p) = e$ if $p = 0$
and $F (p) = f$ if $p > 0$. Set $H (p, m) = \la F (p) \ra^X (m)$. Then $m \in
R \cup S$ iff $\ex p H (p, m) \downarrow$. If $\{ h \} = H$, then $\ex p H (p,
m) \equi H (\tmop{Sel}^X (h, m), m, X)$. Let $\pa{s}^X (m) \simeq \tmop{Sel}^X
(h, m)$. Then $R \cup S = \tmop{dom} \left( \pa{s}^X \right)$. 

(b) This follows from Lemma \ref{existclosure}. (c) EXercise.{\qed}

\begin{lemma}
  A relation $R$ is computable in $X$ iff both $R$ and $\neg R$ are f.~semi-computable in $X$. Moreover a code for $R$ as a f.~computable relation, is then
  effectively obtainable from any pair of codes $(e_0,e_1)$ for $R$ and $\neg R$ as f.~semi-computable relations.
  \end{lemma}

{\pf}$\left( \Imp \right)$: If $R$ is f.~computable in $X$, so is $\neg R$,
and both are f.~semi-computable in $X$.

$(\Leftarrow)$: As both $R$ and $\neg R$ are the domains of some partial f.~computable in $X$ functions, we can choose indices $e_0, e_1$, modifying those
partial computable functions if need be, so that

\hspace{0.7em}$R (m, X) \equi \pa{e_0}^X (m, X) \downarrow 1$

$\neg R (m, X) \equi \pa{e_1}^X (m, X) \downarrow 0.$

\nod Let $f$ be computable with $f (0) = e_0$, $f (1) = e_1$ and $f (k) \uparrow$ for
$k > 1$. Let $G (i, m) = \pa{f (i)}^X (m)$. Suppose that $G = \pa{g}^X$. Then
for all $(m, X)$ $K (m) =_{\tmop{df}} G (\tmop{Sel}^X (g, m), m) )$
is defined, and then $K$ is the total computable characteristic function of $R$.
Hence the latter is f.~c. in $X.$ {\qed}

\begin{lemma}
  For any f.~partial function F,
  
  (i) F is f.~semi-computable in $X$ iff $\tmop{Gr} (F)$ is f.~semi-computable
  in $X$.
  
  (ii) F is f.~computable in $X$ iff $F$ is total and $\tmop{Gr} (F)$ is f.
  computable in $X$.
\end{lemma}

{\pf}Straightforward.

\

We are also going to have an {\em effective choice principle}, a straightforward application of the Selection Theorem.

\begin{theorem}
If $R(p,\vec m, X)$ is a f.~semi-computable relation, and if \linebreak $\all\vec m\all X\ex p R(p,\vec m, X)$, there there is a  f.~semi-c. choice function $F$ so that \\$\all\vec m\all X R(F(\vec m,  X),\vec m,\vec X)$.
\end{theorem}

\begin{lemma}
  For any function $X \in \bai$ the f.~semi-computable in $X$ relations form a
  Spector class.
\end{lemma}

{\pf} To form a Spector class the relations must be closed under $\wedge,
\vee$, bounded and unbounded existential and universal number quantification,
$\ex^{\omega}_{\leq}, \ex^{\omega}, \all^{\omega}_{\leq}, \all^{\omega} .$ It
must have $\omega$-parametized universal classes in each arity and have the
Prewellordering Property. These have now all been established above. \qed

\section{\normalsize The equivalence of feedback computability with Kleene recursion in $ ^2\mathsf{E}$}

We now prove the two Theorems \ref{ftok} and \ref{ktof} which demonstrate directly the
equivalence of the two kinds of
recursion. Firstly that Feedback recursion is an example of Kleene Recursion
in  Theorem \ \ref{ftok}.

\begin{lemma}
  \label{FtoK}(Feedback to Kleene) There is a p.r. function $k_1 \in \bai$ so that
  for any $X \in \bai$, any $e, m \in \omega$
  
  {\hspace{9em}}$\{ k_1 (e) \} ^\me (m, X) \simeq \pa{e}^X (m)$.
\end{lemma}

{\pf} We devise a Kleene partial recursive function $F (h, e, m, X) $ which acts to simulate $\pa{e}^X (m)$. \ Any query initiating a f.~subcomputation call $? Q^X (e_i, m_i) ?$ which occurs during a run of
$\pa{e}^X (m) $ receives back a $\da/ \ua$ response, which we'll here identify with $0
/ 1$ and we shall have $F (h, e, m, X)$ mirror these $0 / 1$ responses, as
outputs from applications of  a Kleene computation.

$F (h, e, m, X)$ proceeds by running a simulation of $\pa{e}^X (m)$, using
the non-query f.c. instructions coded in $\pa{e}^X$ but considered as Kleene
recursion instructions - after all, they can be considered identical. $F$
writes the output from this copy of $\pa{e}^X (m)$ to its own output tape.
When the simulation reaches a query $? Q^X (\bar{e}, \bar{m})$? it replaces it
with a Kleene style query $? Q  (g, h, (\bar{e}, \bar{m}), X) ?$ constructed
as follows.

Let $g_0$ be the index for the Kleene recursive in partial function $\{ g_0 \}=\{ g_0 \}^{\me} $ of Lemma \ref{Kconvergence}. We use that in the following.

Recall that $\{ g_0 \}  (h, (\bar{e}, \bar{m}), X, k)$ will simulate a run of computation of  \linebreak $\{
h \}  ((\bar{e}, \bar{m}), X) $ for $k$ steps.  If after the $k$ steps are completed this simulation of $\{
h \}  ((\bar{e}, \bar{m}), X) $ has halted, then  $\{ g_0 \}  (h,
(\bar{e}, \bar{m}), X, k)\da 0$. Otherwise, and if illfoundedness of \linebreak $\mathfrak{T}^\me( (h ,  \bar{e}, \bar{m}), X)$ has not been encountered, we have that $\{ g_0 \}  (h, (\bar{e},
\bar{m}), X, k)\da 1$. 

By the Kleene Recursion Theorem, let $\bar{h}$ be such that $F (\bar{h}, e, m,
X) \simeq \linebreak \{ \bar{h} \}  ((e, m), X).$

\

{\tmem{Claim}}: \ For any $X \in \can$ any $e, m \in \omega$

\nod (2){\hspace{9em}}$\{ \bar{h} \} ^\me ((e, m), X) \simeq \pa{e}^X (m)$.

In particular

\nod (3)\quad$\ex k$ $\{ \bar{h} \} ^\me ((e, m), X)\da$ in $k$ steps $\Equi  \ex k
\,\,\pa{e}^X (m)\da$ in $k$ steps.

\

{\pf} By induction on the feedback rank $\rho = \rho^X$. If $\rho (e, m) =
0$, then \ $\pa{e}^X (m)$ makes no query calls, and then neither does $F
(\bar{h}, e, m, X) \simeq \linebreak  \{ \bar{h} \}^\me  ((e, m), X)$, and the two outcomes on
their respective output tapes are identical. Suppose then as inductive
hypothesis that (2), and (3), hold for all $(e', m')$ with $\rho (e', m') <
\alpha$. Let $\rho (e, m) = \alpha$. Then any query $? Q^X (e_i, m_i) ?$
occurring during a run of $\pa{e}^X (m)$ has $\rho (e_i, m_i) < \alpha$. Hence
we have (2) and (3) holding for $((e_i, m_i), X)$.

Now running $F (\bar{h}, e, m, X) = \{ \bar{h} \}  ((e, m), X)$,
 we shall have that, by (3) and definition of $\{
g \} $, that for any query $? Q^X (e_i, m_i) ?$ occurring in $\{ \bar{h} \} 
((e, m), X)$'s simulation of $\pa{e}^X (m)$, that,  with $y= \lambda k. \{ g \}  (\bar{h}, (e_i, m_i), X, k)$:\\

{\em $? Q  (g, \bar{h} $, $(e_i, m_i), X$)$?$ returns $E (y) = 0 / 1$ to
the computation $\{ \bar{h} \}  ((e, m), X)$ if and only if $? Q^X (e_i,
m_i) ?$ returns $\da/\ua$ to the computation $\pa{e}^X (m)$.}\\

As apart from queries, the Turing programmes are the same on both sides, the
outcomes of $\{ \bar{h} \}^\me  ((e, m), X)$ and $\pa{e}^X (m)$ are the same, and
hence (2) and (3) hold for $((e, m), X)$.{\qed} {\tmem{Claim}}

\

For the Lemma's statement we let $k_1$ be some p.r. function so that \\ $\{ k_1
(e) \}  (m, X) \simeq \{ \bar{h} \}  ((e, m), X)$. \qed (Theorem \ref{ftok})

\

Secondly, we have that in fact, feedback computation can emulate the specific
case of Kleene recursion in $^2\mathsf{E}$.

\begin{lemma}
  \label{KtoF}(Kleene to Feedback) There is a p.r. function $k_0 \in \bai$ so that
  for any $X \in 
  \bai$ any $e, m \in \omega$
  $$\pa{k_0 (e)}^X (m) \simeq \{ e \} ^\mathsf{^2 E} (m, X).$$
\end{lemma}

{\pf}We define a f.~semi-comp. in $X$ function $F (\bar{e}, e, m) = F^X
(\bar{e}, e, m)$ which acts as follows.

$F$ runs as $\{ e \}  ^\mathsf{ E} (m, X)$ using the underlying Turing program until a
query $Q^{\mathsf{E}}  (e_0, m_0, X)$ is reached. Recall that this Kleene query then computes \linebreak  $\{e_0\}^\mathsf{E}((m_0,k),X)\da
y(k)$ for increasing $k=0,1,\ldots$ 

\

Now let $P$ be the f.c. procedure that queries $? \la \bar{e} \ra^X
(( e_0, ( m_0 , k))) \da / \ua$? in turn,
starting with $k = 0$, then by incrementing $k : = k + 1$ when it learns
that $\la \bar{e} \ra^X (( e_0, ( m_0 , k)))
\da$. If for some (least) $k_0$, $\la \bar{e} \ra^X (e_0, (
m_0 , k_0) ) \ua$,  $P$ then, on receipt of this information, invokes some
self-referential program for example, in order to ensure that $P\Uparrow$, that is,
becomes undefined.\\

$F$ then asks the following query (still in the language of feedback
computation):\\

$Q_1$``{\tmem{ If }}$P${\tmem{ is the f.c. procedure described above, does}}
$P \da /\ua$?''\\

If, for some $k$, $\la \bar{e} \ra^X (( e_0, ( m_0
, k))) $ $\Uparrow$, then the query remains unanswered. If
for some $k$ $\la \bar{e} \ra^X (( e_0, ( m_0 ,
k))) \ua$  then we have also ensured that $F (\bar{e}, e, m) \Uparrow$
as no response is received in $Q_1$. Otherwise $P$ cannot halt.

However, on the answer ``$\ua$'' to $Q_1 \tmop{we}$ then know that $\la \bar{e}
\ra^X (( e_0, ( m_0 , k)))\downarrow$ for all $k$. $F$ then
proceeds to ask:

\

$Q_2$ ``{\tmem{Let $P_1$ be the f.c. procedure that computes $\la \bar{e}
\ra^X (( e_0, ( m_0 , k)))\da y (k$) and so that}}

{\tmem{$P_1\da 0$, {\ie}halts with $0$, if }}$ y (k) = 0$.

\tmem{Otherwise if $y (k) \neq 0$, $P_1$ increments $k : = k + 1$ and then
computes }

$\la \bar{e} \ra^X (( e_0, ( m_0 , k + 1)))
= {y(k+1)}$; \tmem{ again halting if } $y(k+1)=0$ { \em and otherwise returning and incrementing to } $k+2$ \etc {\em Does }$P_1 \da / \ua $?''

\

Comment: If the answer to $Q_2$ is ``$\ua$'', then, setting $y = (\vec{y (k)})$, we have that $E (y) = 1. \quad \tmop{If} Q_2$ answers
``$\da$'' then $E (y)$ would be $0$, if we calculated all the values $y (k),$ and
hence these f.~c. responses mirror those from  of the Kleene recursion.

(At first glance $P $ might look superfluous: why does $P_1$ as a query not
suffice? However we might have $\la \bar{e} \ra^X (( e_0, (
m_0 , k))) \da 0$ whilst $\la \bar{e} \ra^X (( e_0,
( m_0 , k + 1))) {\Uparrow}$. We do need to
ask first whether $\la \bar{e} \ra^X (( e_0, ( m_0
, j))) \da $ for all $j$.)

After this, $F$ runs as $\{ e \}  ^\mathsf{ E} (m, X)$ using the underlying Turing program until the next 
query (if any) $Q^{\mathsf{E}}  (e_1, m_1, X)$ is reached. It then acts as above asking the same queries now about the pair $e_1,m_1$. And so forth.

\

By the Feedback Recursion Theorem there is $\tilde{e}$ so that $F (\tilde{e}, e, m)
= \la \tilde{e} \ra^X (e , m)$. We then take $k_0$ so that for any $e, \pa{k_0
(e)}^X (m) = \la \tilde{e} \ra^X (e , m)$.

\

\nod{\tmem{Claim:}} $k_0$ {\tmem{satisfies the Lemma.}}

\

{\pf}Note first that $\pa{k_0
(e)}^X (m) $ halts if and only if $\{ e \} ^\mathsf{^2 E} (m, X)$ does, The reason being that, by design, the  $\pa{k_0
(e)}^X (m) $ computation is getting the same information from its oracle for $\{\da,\ua\}$, once identifying these with $\{0,1\}$ respectively, as the  $\{ e \} ^\mathsf{^2 E} (m, X)$ does. The rest of the Turing instructions are the same.  Hence the identical halting behaviour.

 The {\em Claim } can be proven by induction on the Kleene rank $\rho^\me $. If $\rho ^{\me}
(e, m, X)$ is zero, then $\la \tilde{e} \ra^X (e , m)$ makes no query calls
and so the lemma holds easily; if, as an inductive hypothesis, the lemma holds
for all $e_0, \vec{m}_0, X$ with $\rho ^{\me} (e_0, \vec m_0 , X) <
\alpha$, and $\rho^{\me}  (e, m, X) = \alpha$, then all query calls of the form $Q ^{\me}
(e_0, (m_0, k), X)$ occurring in $\{ e \} ^\mathsf{^2 E} (m, X)$, are faithfully translated from being about $\{ e_0 \} ^{\me}
((m_0, k), X)$ to being about  $\pa{k_0 (e_0)}^X ((m_0, k))$, as the former are
of smaller Kleene rank than $\alpha$. \ If for some $k$, $\{ e_0 \} ^{\me} ((m_0,
k), X)\ua$ then so does $\la \tilde{e} \ra^X (e_0, ( m_0 ,
k$))$\Uparrow$, and both $\pa{e}^X (m) $ and $ \{ e \} ^{\me} (m, X)$ are
undefined. Otherwise by the reasoning of the last paragraph, for all $k$ $\la \tilde{e} \ra^X
(e_0, ( m_0 , k)) \simeq \{ e_0 \}^\me  ((m_0, k), X)$, with
both sides defined.{\qed} (Theorem \ref{ktof})

\bibliographystyle{plain}
\bibliography{settheory10v}

\ed

\end{document}